\documentclass[11pt]{amsart}
\usepackage[a4paper,margin=2.4cm]{geometry}
\usepackage{mathtools,amsmath,amsthm,amssymb,amscd,amsfonts,faktor, tikz, caption, subcaption}
\usepackage{booktabs}
\usepackage{thm-restate}
\usepackage{hyperref}
\usepackage{easyReview}

\theoremstyle{plain}
\newtheorem{theorem}{Theorem}[section]
\newtheorem{proposition}[theorem]{Proposition}
\newtheorem{lemma}[theorem]{Lemma}

\numberwithin{equation}{section}

\theoremstyle{definition}
\newtheorem{question}[theorem]{Question}
\newtheorem{remark}[theorem]{Remark}
\newtheorem{example}[theorem]{Example}

\newcommand{\Z}{\mathbb{Z}}

\newcommand{\R}{\mathbb{R}}
\newcommand{\cF}{\mathcal{F}}

\newcommand{\cM}{\mathcal{M}}
\newcommand{\cN}{\mathcal{N}}
\newcommand{\cK}{\mathcal{K}}
\newcommand{\cL}{\mathcal{L}}
\newcommand{\cV}{\mathcal{V}}
\newcommand{\cS}{\mathcal{S}}
\newcommand{\fK}{\mathfrak{K}}
\newcommand{\fL}{\mathfrak{L}}

\renewcommand{\subset}{\subseteq}
\DeclareMathOperator{\lk}{Lk}
\DeclareMathOperator{\proj}{Proj}
\DeclareMathOperator{\wed}{Wed}

\title[Toric manifolds with Picard number 4]{Complete non-singular toric varieties with Picard number 4}

\author{Suyoung Choi}
\address{Department of mathematics, Ajou University, 206, World cup-ro, Yeongtong-gu, Suwon 16499,  Republic of Korea}
\email{schoi@ajou.ac.kr}
\author{Hyeontae Jang}
\address{Department of mathematics, Ajou University, 206, World cup-ro, Yeongtong-gu, Suwon 16499, Republic of Korea}
\email{a24325@ajou.ac.kr}
\author{Mathieu Vall\'ee}
\address{Universit\'e Sorbonne Paris Nord, LIPN, CNRS UMR 7030, F-93430, Villetaneuse, France}
\email{vallee@lipn.fr}

\date{\today}


\thanks{This research was supported by Basic Science Research Program through the  National Research Foundation of Korea(NRF) (RS-2021-NR060141 and RS-2025-00521982).}

\begin{document}
\begin{abstract}
We classify all complete non-singular toric varieties with Picard number four via a combinatorial framework based on fanlike simplicial spheres and characteristic maps.
This classification yields $59$ fanlike seeds with Picard number four, along with all toric manifolds supported by them.
As a consequence, we resolve a conjecture of Gretenkort, Kleinschmidt, and Sturmfels by presenting the first known examples of toric manifolds supported by neighborly polytopes.
We also answer a question of Batyrev concerning minimal non-faces of such spheres.
\end{abstract}
\maketitle

\section{Introduction}
The classification of algebraic varieties is a fundamental problem in algebraic geometry.
Among these, toric varieties serve as a bridge between algebraic geometry, combinatorics, and topology; see \cite{Buchstaber-Panov2015, Cox-Little-Schenck2011, fulton1993toric, Oda1988} for details.
Their explicit and computable nature makes them valuable as testbeds for illustrating theoretical phenomena and verifying conjectures.

Despite their rich combinatorial structure, the complete classification of complete non-singular toric varieties, simply referred to as \emph{toric manifolds}, remains an open and challenging problem, especially in higher dimensions.
While all complex two-dimensional toric manifolds are well-known to be obtained through successive blow-ups of either the projective plane or Hirzebruch surfaces \cite{Orlik-Raymond1970, Oda1988}, an explicit classification in higher dimensions has yet to be fully understood.

An alternative approach to classification considers the \emph{Picard number}, which measures the rank of the Picard group.
The classification of toric manifolds has been successfully completed up to Picard number~$3$. In their pioneering works, Kleinschmidt \cite{Kleinschmidt1988} classified those with Picard number~$2$, and Batyrev \cite{Batyrev1991} extended the classification to Picard number~$3$.
The \emph{fundamental theorem for toric geometry} establishes a bijection between complex $n$-dimensional toric manifolds with Picard number~$p$ and real $n$~dimensional complete non-singular fans with $n+p$ rays.
Each of these fans is characterized by two components:
\begin{itemize}
  \item an $(n-1)$-dimensional piecewise linear (for short, PL) sphere, called its \emph{underlying complex}, and
  \item a set of $n+p$ primitive ray vectors in $\Z^n$, one for each vertex of the sphere.
\end{itemize}
More precisely, for an $n$-dimensional complete non-singular fan with $n+p$ rays, each cone corresponds to a subset of the rays, and the collection of such subsets forms a sphere.
For clarity, we call a PL~sphere \emph{fanlike} if it arises as the underlying complex of a complete non-singular fan.
Therefore, the classification of complete non-singular toric varieties reduces essentially to the following question;
\emph{Which pairs of fanlike PL~spheres and sets of primitive ray vectors produce complete non-singular fans?}

However, this problem remains highly nontrivial since, even after fixing either the dimension or the Picard number, infinitely many fanlike spheres exist.
A complete characterization of fanlike spheres is not yet known, presenting a significant barrier to classification.

Numerous studies such as \cite{Batyrev1999, Delaunay2005, Cox-von2009, Suyama2014, Suyama2015, Bogart-Haase-Hering2015, Baralic-Milenkovic2022, Ayzenberg2024} have investigated necessary conditions for a PL~sphere to be fanlike, but complete characterizations have remained out of reach, leaving many important open questions and conjectures.
In particular, Batyrev \cite{Batyrev1991} suggested an upper bound on the number of minimal non-faces of a fanlike sphere, and Gretenkort, Kleinschmidt, and Sturmfels \cite{Gretenkort-Kleinschmidt-Sturmfels1990} conjectured that no neighborly fanlike sphere exist.
Remarkably, this paper resolves both issues, including providing a counterexample to the conjecture of  Gretenkort, Kleinschmidt, and Sturmfels.
The fact that their conjecture stood unresolved over three decades despite being incorrect underscores the inherent difficulty in classifying toric manifolds.

A significant breakthrough toward classification was achieved by Choi and Park \cite{Choi-Park2016}.
While the \emph{wedge operation} was already known, Choi and Park were the first to systematically apply this operation to classify toric manifolds.
They showed that if all toric manifolds supported by a given PL~sphere~$K$ are known, then applying a wedge operation at a vertex of $K$ allows one to determine all toric manifolds supported by the resulting complex, which has one dimension higher while preserving the Picard number.
This approach provides a highly practical and concrete construction.
By iteratively applying this method, the classification of toric manifolds supported by~$K$ extends to include all toric manifolds supported by the PL~spheres obtained through successive wedge operations on~$K$.
These resulting PL~spheres form what is known as the \emph{$J$-construction}, introduced by Bahri, Bendersky, Cohen, and Gitler \cite{BBCG2015}, and are denoted by~$K(J)$, where $J$ is a tuple of positive integers specifying the number and positions of the wedge operations.
If a PL~sphere cannot be obtained through wedge operations, we call it a \emph{seed}.
Notably, any PL~sphere can be expressed in the form $K(J)$ with $K$ being a seed.
Thus, the classification of toric manifolds with Picard number~$p$ essentially reduces to identifying all fanlike seeds with Picard number~$p$ and determining the toric manifolds supported by them.

In their subsequent work \cite{CP_wedge_2}, Choi and Park proved that for a fixed Picard number, there exist only finitely many fanlike seeds.
For example, the only fanlike seeds with Picard number~$3$ are the boundaries of a pentagon and a $3$-crosspolytope, while the only seed with Picard number~$2$ is the boundary of a square.
These results provide immediate reproofs of Kleinschmidt and Batyrev's classification results; see, for instance, \cite[Theorem~6.8]{Choi-Park2016}.
A major advantage of this approach is that it does not rely on any geometric properties of toric manifolds, whereas Batyrev's method applies only to projective toric manifolds.
Since non-projective toric manifolds exist for Picard number at least ~$4$, Batyrev's method is not applicable in this case.
However, the approach based on Choi and Park's work remains viable.
Indeed, this paper successfully applies it to classify all toric manifolds with Picard number~$4$.

The authors of this paper previously classified all potential candidates for fanlike seeds with Picard number~$4$ in \cite{Choi-Jang-Vallee2024}.
They identified all seeds with Picard number~$4$ that support a non-singular characteristic map, a necessary condition for being fanlike, and justified that there are exactly $3153$ such seeds, and included them in a publicly accessible database.

In this paper, we provide a complete characterization of the fanlike seeds from this list, presented in Theorem~\ref{thm:characterization_fanlike}.
Using this characterization, it is established that there exist exactly $59$ fanlike seeds; see Section~\ref{section: fans list}.
Furthermore, all toric manifolds supported on these fanlike seeds are identified.
As a direct consequence of these results and Choi-Park's framework, these lists complete the full classification of toric manifolds with Picard number~$4$.

As highlighted earlier, our results provide explicit counterexamples to previously unresolved conjectures posed by Batyrev and Gretenkort, Kleinschmidt, and Sturmfels, offering negative answers to their questions in the final section of this paper; see Propositions~\ref{prop:answer_Batyrev} and~\ref{prop:ans_GKS}.
Such explicit classification results not only resolves a long-standing open problem but also highlights the effectiveness of our classification framework in addressing foundational questions in toric geometry.

\section{Classification of toric manifolds through wedge operations}
Let $K$ be an $(n-1)$-dimensional PL~sphere on $[m]=\{1, 2, \dots, m\}$.
A \emph{(non-singular) characteristic map} $\lambda$ over $K$ is a map from $[m]$ to $\Z^n$ such that for each facet~$\sigma$ of $K$, $\lambda(\sigma)$ forms a basis of $\Z^n$.
The pair~$(K, \lambda)$ is called a \emph{characteristic pair}.
This pair is referred to as \emph{fan-giving} if there exists a complete non-singular fan such that its underlying complex is $K$, and each $1$-cone is generated by the vector~$\lambda(v)$ for some vertex~$v$ of~$K$.
In this case, $K$ is called \emph{fanlike}, and $\lambda$ is also said to be \emph{fan-giving}.

We often regard $\lambda$ as the matrix
$\begin{bmatrix}
    \lambda(1) & \lambda(2) & \cdots & \lambda(m)
\end{bmatrix},$
and $\lambda_\sigma$ denotes the submatrix of~$\lambda$ with columns~$\{\lambda(v) \mid v \in \sigma\}$ for each face~$\sigma$ of~$K$.
For consistency, we use $\lambda_v$ instead of~$\lambda(v)$ for each vertex~$v$.
Two characteristic maps over $K$ are called \emph{Davis-Januszkiewicz equivalent}, or simply \emph{D-J equivalent}, if one can be obtained from the other by left multiplication with an invertible integer matrix.
See \cite{Davis-Januszkiewicz1991} for details.
Two fan-giving characteristic maps over $K$ are \emph{isomorphic} if one is D-J equivalent to the other after right multiplication with the permutation matrix induced by an automorphism of $K$.

It is known that if $K$ is fanlike, then $K$ must be starshaped, thus it is a PL~sphere~\cite{Buchstaber-Panov2015}.
However, not every PL~sphere is fanlike.
For example, the boundary of a $4$-dimensional cyclic polytope with seven vertices is not the underlying complex of any complete non-singular fan \cite{Gretenkort-Kleinschmidt-Sturmfels1990}.
Furthermore, even if $K$ is fanlike, not every characteristic map over $K$ is fan-giving.

The \emph{link} of $\sigma$ in $K$ is the simplicial complex: $$\lk_K(\sigma) = \{\tau \in K \mid \tau \cup \sigma \in K, \, \tau \cap \sigma = \varnothing  \}.$$
Let $\lambda$ be a characteristic map over $K$.
The \emph{projected characteristic map}, or simply \emph{projection}, of~$\lambda$ with respect to~$\sigma$ is defined as the composition of two maps:
$$
    \proj_\sigma(\lambda) \colon [m] \setminus \sigma \xrightarrow{\lambda \vert_{[m] \setminus \sigma}} \Z^n \longrightarrow \Z^n /\left< \lambda(w) \mid w \in \sigma \right>  \cong \Z^{n-k},
$$
where the latter map is the canonical projection.
\begin{proposition}[\cite{Hattori-Masuda2003}, \cite{Choi-Park2016}] \label{proposition: link fan-giving}
  Let $(K, \lambda)$ be a fan-giving characteristic pair.
  Then, $(\lk_K(\sigma), \proj_\sigma(\lambda))$ is a fan-giving characteristic pair for any face $\sigma$ of $K$.
\end{proposition}

The \emph{wedge} of $K$ at a vertex $v$ is defined as the simplicial complex
$$
    \wed_v(K) = \{\sigma \cup \tau \mid (\sigma \in I, \, \tau \in \lk_K(v)) \text{ or } (\sigma \in \partial I, \, v \not \in \tau \in K)\},
$$
where $I$ is the $1$-simplex with vertex set $\{v, v'\}$, and $\partial I$ is its boundary complex.
In other words, $\wed_v(K)$ is obtained from $K$ by replacing the vertex $v$ with two distinct vertices $v$ and $v'$, connecting each to the faces not containing $v$, and ``wedging'' along the link of~$v$.

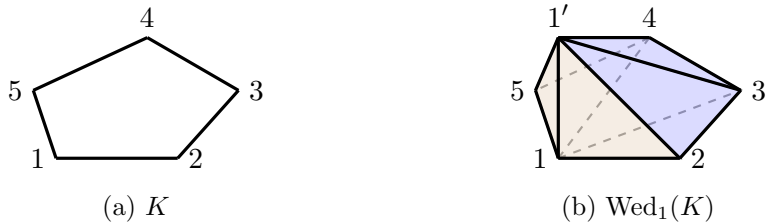
\begin{figure}[h]
    \centering
        \begin{subfigure}[t]{0.4\textwidth}
        \centering
        \begin{tikzpicture}[thick,scale=2]
            \coordinate[label=left:1] (1) at (0, 0);
            \coordinate[label=right:2] (2) at (0.8, 0);
            \coordinate[label=right:3] (3) at (1.2, 0.45);
            \coordinate[label=above:4] (4) at (0.6, 0.8);
            \coordinate[label=left:5] (5) at (-0.15, 0.45);

            \draw[very thick] (1) -- (2);
            \draw[very thick] (2) -- (3);
            \draw[very thick] (3) -- (4);
            \draw[very thick] (4) -- (5);
            \draw[very thick] (1) -- (5);
        \end{tikzpicture}
        \subcaption{$K$}
        \end{subfigure}
        \begin{subfigure}[t]{0.4\textwidth}
        \centering
        \begin{tikzpicture}[thick,scale=2]
            \coordinate[label=left:$1$] (1) at (0, 0);
            \coordinate[label=right:2] (2) at (0.8, 0);
            \coordinate[label=right:3] (3) at (1.2, 0.45);
            \coordinate[label=above:4] (4) at (0.6, 0.8);
            \coordinate[label=left:5] (5) at (-0.15, 0.45);
            \coordinate[label=above:$1'$] (6) at (0, 0.8);

            \draw[dashed] (1) -- (3);
            \draw[dashed] (1) -- (4);
            \draw[dashed] (4) -- (5);

            \draw[fill=brown!20,opacity=0.7] (1) -- (2) -- (6);
            \draw[fill=brown!20,opacity=0.7] (1) -- (5) -- (6);
            \draw[fill=blue!20,opacity=0.7] (2) -- (3) -- (6);
            \draw[fill=blue!20,opacity=0.7] (3) -- (4) -- (6);

            \draw[very thick] (1) -- (2);
            \draw[very thick] (2) -- (3);
            \draw[very thick] (3) -- (4);

            \draw[very thick] (1) -- (5);
            \draw[very thick] (1) -- (6);
            \draw[very thick] (2) -- (6);
            \draw[very thick] (3) -- (6);
            \draw[very thick] (4) -- (6);
            \draw[very thick] (5) -- (6);
        \end{tikzpicture}
        \subcaption{$\wed_{1}(K)$}
        \end{subfigure}
        \caption{Wedge of the boundary of a pentagon.}
\end{figure}

One can regard $K$ as a subcomplex of $\wed_v(K)$.
In fact, we have natural isomorphisms
$$
    K = \lk_{\wed_v(K)}(v') \cong \lk_{\wed_v(K)}(v).
$$
Moreover, it is known that $K$ is a PL~sphere if and only if $\wed_v(K)$ is a PL~sphere for any vertex~$v$ of $K$ \cite{Choi-Jang-Vallee2024}.
These facts suggests that taking the link at either of the distinguished vertices~$v$ or~$v'$ can be viewed, in some sense, as the inverse operation to the wedge operation at the vertex~$v$, within the collection of isomorphism classes of PL~spheres.
This relationship is also reflected in the interaction between the collection of D-J equivalent classes over $K$ and those of~$\wed_v(K)$.

Up to simplicial isomorphism, we may assume that $\{1, 2, \ldots, n\}$ is a facet of~$K$.
Thus, each D-J equivalence class admits a unique representative whose first $n$ columns form the $n \times n$ identity matrix~$I_n$.
Let
\begin{align*}
    \lambda_1 = \begin{bmatrix}
                    1 & \mathbf{0} & \textbf{a} \\
                    \mathbf{0} & I_{n-1} & A
                    \end{bmatrix} \text{ and }
    \lambda_2 = \begin{bmatrix}
                    1 & \mathbf{0} & \textbf{b} \\
                    \mathbf{0} & I_{n-1} & A
                \end{bmatrix}
\end{align*}
be characteristic maps over~$K$, where $\mathbf{0}$ denotes zero vectors of appropriate size and position, and~$A$ is an $(n-1) \times (m-n)$ matrix.
Then one can directly check that
\begin{equation} \label{eq: wedged map}
    \lambda =   \begin{bmatrix}
                    1 & 0 & \mathbf{0} & \textbf{a} \\
                    0 & 1 & \mathbf{0} & \textbf{b} \\
                    \mathbf{0} & \mathbf{0} & I_{n-1} & A
                \end{bmatrix}
\end{equation}
is a characteristic map over $\wed_{1}(K)$, where the first two columns correspond to the vertices~$1$ and~$1'$.
Note that the projections~$\proj_{1}(\lambda)$ and~$\proj_{1'}(\lambda)$ recover $\lambda_2$ and~$\lambda_1$, respectively.
Conversely, given a characteristic map~$\lambda$ over $\wed_{1}(K)$, the presence of the edge~$\{1, 1'\}$ implies the existence of a facet containing this edge.
Thus, we may assume that $\lambda$ has the form~\eqref{eq: wedged map}, where the first two columns correspond to vertices~$1$ and~$1'$.

This observation shows that any characteristic map over $\wed_v(K)$ can be reconstructed from two characteristic maps defined over $K$.
Moreover, the following holds.
\begin{theorem} [\cite{Choi-Park2016}] \label{theorem: wedge fan-giving}
  Let $K$ be a PL~sphere, and $\lambda$ a characteristic map over $\wed_v(K)$.
  Then, for any vertex $v$ of $K$, $(\wed_v(K), \lambda)$ is fan-giving if and only if $(K, \proj_{v}(\lambda))$ and $(K, \proj_{v'}(\lambda))$ are fan-giving.
\end{theorem}

A PL~sphere is called a \emph{seed} if it cannot be expressed as the wedge of another PL~sphere.
By Theorem~\ref{theorem: wedge fan-giving}, every fanlike sphere~$K$ is obtained through a sequence of wedge operations starting from a seed~$L$.
Furthermore, every fan-giving characteristic map over $K$ is constructed from fan-giving characteristic maps over $L$.
Note that the wedge operation increases both the number of vertices $m$ and the dimension~$n-1$ by exactly one, thus preserving the quantity~$m-n$.
This invariant quantity $m-n$ is called the \emph{Picard number} of the PL~sphere~$K$.
Indeed, it coincides precisely with the Picard number of the toric manifolds supported by~$K$.

\begin{theorem} \cite{Choi-Jang-Vallee2024} \label{theorem: classification}
    Up to simplicial isomorphism, the numbers of seeds of dimension~$n-1$ with Picard number~$p\leq4$ admitting a characteristic map are:
    \begin{center}
    	\begin{tabular}{l*{12}{c}r}
    		\toprule
    		& \multicolumn{12}{c}{n}\tabularnewline  \cmidrule(lr){2-13}
    		$p$ &$1$ & $2$ & $3$ & $4$ & $5$ & $6$ & $7$ & $8$ & $9$ & $10$ & $11$ & $>11$& total \tabularnewline \midrule
    		$1$&$1$&  &&&&&&&&&&&$1$\tabularnewline
    		$2$ && $1$ & & &&&&&&&&& $1$\tabularnewline
    		$3$ && $1$ & $1$& $1$& & & & &&&& &$3$\tabularnewline
    		$4$ && $1$ & $4$ & $21$ & $142$ & $733$ & $1190$ & $776$ & $243$& $39$ & $4$ & &$3153$\tabularnewline
    		\bottomrule
    	\end{tabular}
    \end{center}
   where empty entries represent zero.
\end{theorem}

In the table, the seed with $p=1$, $n=1$ is the simplicial complex consisting of two isolated vertices without edges.
The seed with $p=2$, $n=2$ is the boundary of a quadrilateral.
When $p=3$, the seed for $n=2$, $3$, $4$ are respectively the boundaries of a pentagon~$P^2(5)$, a $3$-crosspolytope, and a $4$-dimensional cyclic polytope~$C^4(7)$ with seven vertices.
Among these, with Picard number at most~$3$, all are fanlike except for the boundary $\partial C^4(7)$ of the $4$-cyclic polytope \cite{Gretenkort-Kleinschmidt-Sturmfels1990}.
Toric manifolds supported by these fanlike seeds are already classified and well-known from the works of Kleinschmidt \cite{Kleinschmidt1988} and Batyrev \cite{Batyrev1991}.

In the remainder of this paper, we provide a complete classification of all fanlike seeds with Picard number~$4$, along with explicit fan-giving characteristic maps over each seed.
As previously discussed in this section, the wedge operation allows us to explicitly construct every toric manifold of Picard number~$4$ from these fanlike seeds.
Therefore, the list provides a complete and explicit classification of all toric manifolds with Picard number~$4$.

\section{Fan-giving characteristic pairs}
Let $K$ be an $(n-1)$-dimensional PL~sphere on $[m]$.
For a $(k-1)$-face~$\sigma$ of $K$, $\{\lambda_v \mid v \in \sigma \}$ generates a strongly convex $k$-cone by the non-singularity condition of~$\lambda$.
We denote the cone by $pos(\sigma)$.
This cone depends on $\lambda$, but we use this notation only in the proof of the following lemma, where $\lambda$ is fixed.

\begin{lemma} \label{lemma: fan-givingness}
    A characteristic pair~$(K, \lambda)$ is fan-giving if and only if for any two facets $\sigma_1$ and $\sigma_2$ of $K$, every nonzero vector $x = [x_v]_{v \in \sigma_1 \cup \sigma_2} \in \ker(\lambda_{\sigma_1 \cup \sigma_2})$
    satisfies at most one of the following;
    \begin{itemize}
        \item $x_v \geq 0$ for all $v \in \sigma_1 \setminus \sigma_2$.
        \item $x_v \leq 0$ for all $v \in \sigma_2 \setminus \sigma_1$.
    \end{itemize}
\end{lemma}

\begin{proof}
  When the set $\{pos(\lambda_\tau) \mid \tau \in K\}$ forms a fan, its underlying simplicial complex is a simplicial sphere $K$, so it is complete.
  By the definition of a fan, the pair is fan-giving if and only if for any two distinct faces $\tau_1$ and $\tau_2$ of $K$, two cones $pos(\tau_1)$ and $pos(\tau_2)$ intersect only at their faces, that is, $pos(\tau_1 \cap \tau_2) = pos(\tau_1) \cap pos(\tau_2)$.

  Note that by the definition of a cone, $$pos(\tau_1 \cap \tau_2) \subset pos(\tau_1) \cap pos(\tau_2).$$

  A point $z \in pos(\tau_1) \cap pos(\tau_2)$ is expressed in two ways:
  \begin{equation} \label{eqn:z_expression}
    z = \sum_{v \in \tau_1} \lambda_v c_v = \sum_{v \in \tau_2} \lambda_v d_v
  \end{equation}
  for some nonnegative coefficients $c_v$ and $d_v$.
  Note that since the columns of $\lambda$ corresponding to a face of $K$ are a part of a basis of $\Z^n$ by the non-singularity condition of $\lambda$, each $c_v$ and $d_v$ are uniquely determined.
  Rearranging the equation, we obtain
  $$
    \sum_{v \in \tau_1 \cap \tau_2} \lambda_v (c_v-d_v) +  \sum_{v \in \tau_1 \setminus \tau_2} \lambda_v c_v + \sum_{v \in \tau_2 \setminus \tau_1} (- d_v) \lambda_u =0.
  $$
  Then, the coefficients $c_v-d_v$, $c_v$, $-d_v$ in the above equation can be regarded as a component of a vector $x \in \ker(\lambda_{\tau_1 \cup \tau_2})$
  satisfying
  \begin{equation}\label{eqn:condition}
   x_v \geq 0  \text{ for all } v \in \tau_1 \setminus \tau_2, \text{ and } x_v \leq 0  \text{ for all } v \in \tau_2 \setminus \tau_1.
  \end{equation}
  By the uniqueness of the expression of $z$ in each of the cones $pos(\tau_1)$, $pos(\tau_2)$, and $pos(\tau_1 \cap \tau_2)$, the following are equivalent:
  \begin{itemize}
    \item $z \in pos(\tau_1 \cap \tau_2)$,
    \item $x_v = 0$ for all $v \in \tau_1 \setminus \tau_2$,
    \item $x_v = 0$ for all $v \in \tau_2 \setminus \tau_1$,
    \item $x_v = 0$ for all $v \in \tau_1 \cup \tau_2$.
  \end{itemize}
  Hence if there is no nonzero vector $x \in \ker(\lambda_{\tau_1 \cup \tau_2})$ satisfying \eqref{eqn:condition}, then $$pos(\tau_1) \cap pos(\tau_2) \subset pos(\tau_1 \cap \tau_2)$$

  Conversely, assume that $pos(\tau_1 \cap \tau_2) = pos(\tau_1) \cap pos(\tau_2)$ holds.
  Suppose that $x$ is a nonzero element of the kernel satisfying \eqref{eqn:condition}.
  By moving negative terms in $x$ to the other side of the equation, we obtain a point $z \in  pos(\tau_1) \cap pos(\tau_2)$ in the sense of \eqref{eqn:z_expression}.
  Since $z \in pos(\tau_1 \cap \tau_2) = pos(\tau_1) \cap pos(\tau_2)$, by the above equivalence, $x$ is the zero vector, which contracts the assumption, and, hence, any nonzero element of the kernel does not satisfy \eqref{eqn:condition}.

  Finally, we need to show that we can only consider facets.
  Assume that the condition in the statement of the lemma holds.
  Let $\tau_1$ and $\tau_2$ be faces of $K$.
  Suppose that there is a nonzero vector $z \in (pos(\tau_1) \cap pos(\tau_2)) \setminus pos(\tau_1 \cap \tau_2)$.
  Without loss of generality, by the above equivalence, we may assume that $c_{v'}$ is positive for a vertex $v' \in \tau_1 \setminus \tau_2$.
  Let $\sigma_1$ and $\sigma_2$ be facets of~$K$ containing~$\tau_1$ and~$\tau_2$, respectively.
  Since $K$ is a PL~sphere, every $(n-2)$-face is contained in exactly two facets.
  Thus, we can choose $\sigma_2$ not containing $v'$.
  Let $x$ be the vector in $\R^{\sigma_1 \cup \sigma_2}$ such that $$x_v = \begin{cases}
                                            c_v-d_v, & \mbox{if } v \in \tau_1 \cap \tau_2, \\
                                            c_v, & \mbox{if } v \in \tau_1 \setminus \tau_2 \\
                                            -d_v, & \mbox{if } v \in \tau_2 \setminus \tau_1 \\
                                            0, & \mbox{otherwise}.
                                          \end{cases}$$
  Then $x \in \ker(\lambda_{\sigma_1 \cup \sigma_2})$, and $x$ satisfies \eqref{eqn:condition} with $\sigma_1$ and $\sigma_2$ instead of $\tau_1$ and $\tau_2$.
   By the assumption and the above equivalence, $x_v = 0$ for all $v \in \sigma_1 \cup \sigma_2$.
  Therefore, $c_{v'}=0$, which leads to a contradiction.
\end{proof}

\begin{remark}
  The condition in Lemma~\ref{lemma: fan-givingness} can be reduced to comparing signed circuits of the oriented matroid associated to $\lambda$.
  However, we need to compute fan-givingness of not only fixed characteristic maps, but also characteristic maps with some variables, see Section~\ref{section: fans list}.
  Therefore, as the oriented matroid associated to $\lambda$ is not fixed, no simple algorithm exist for solving this problem.
\end{remark}

For the vertex $1$ of $K$, we consider the situation in which all fan-giving characteristic maps over $\lk_K(1)$ are already known.
Up to D-J equivalence, $\lambda$ can be written in the form:
$$
    \begin{bmatrix}
        1 & \mathbf{0} & b_1 & \dots & b_{m-n} \\
        \mathbf{0} & I_{n-1} & \multicolumn{3}{c}{A}
    \end{bmatrix},
$$ where $b_1, \ldots, b_{m-n}$ are some integers, $A$ is an $(n-1) \times (m-n)$ matrix, and $\mathbf{0}$ denotes zero vectors of appropriate sizes at the corresponding positions.
We emphasize that the symbols $A$, $b_1, \ldots, b_{m-n}$ represent unknown integers.
The submatrix of
$\begin{bmatrix}
    I_{n-1} & A
\end{bmatrix}$
consisting of the columns corresponding to the vertices of $\lk_K(1)$ is a characteristic map over $\lk_K(1)$.
According to Proposition~\ref{proposition: link fan-giving}, for~$\lambda$ to be fan-giving, this characteristic map given by the submatrix must also be fan-giving.
Then the information about the fans over $\lk_K(1)$ restricts a large portion of the possible values of unknowns in $A$ if the link contains many vertices.
In particular, if $\lk_K(1)$ has $m-1$ vertices, that is, $\begin{bmatrix}
                                                             I_{n-1} & A
                                                           \end{bmatrix}$ itself is a characteristic map over $\lk_K(1)$, then there remain only $m-n$ unknowns $b_1$, \dots, $b_{m-n}$.

Fix a fan-giving characteristic map over $\lk_K(1)$.
To determine $\lambda$ explicitly, we proceed as follows.
Let $\cF$ be the facet set of $K$.
We assign a natural ordering on the vertex set $[m]$ of $K$, which induces a map $d \colon \cF \to \{-1, 1\} \subset \Z$, defined by
$$
    d(\sigma) = \det(\lambda_\sigma).
$$
On the other hand, once an orientation of $K$ is fixed, the chosen ordering induces another map
$$
    o \colon \cF \to \{-1, 1\} \subset \Z.
$$
If these two maps coincide, i.e., if $d(\sigma) = o(\sigma)$ for all $\sigma \in \cF$, then $(K, \lambda)$ is said to be \emph{positively oriented}, or simply \emph{positive}.
When the simplicial complex $K$ is clear from the context, we call $\lambda$ itself positive.
\begin{proposition}[\cite{Choi-Park2016}] \label{proposition: fan-giving positive}
  If $(K, \lambda)$ is a fan-giving characteristic pair, then it is positive.
\end{proposition}

We choose a representative of~$\lambda$ such that the first $n$ columns form the identity matrix.
By Proposition~\ref{proposition: fan-giving positive}, this choice determines the orientation of $K$ and the map $o$ by setting
$$
    o(\{1, \dots, n\})=d(\{1, \dots, n\})=\det(I_n)=1,
$$
in accordance with fan-givingness.
Next, we explicitly construct positive characteristic maps by determining the remaining unknown parameters such that the map~$d$ matches the orientation map~$o$ on all facets.
These resulting maps no longer contain unknown parameters though they might still involve indeterminates.
Finally, by applying Lemma~\ref{lemma: fan-givingness}, we verify whether these explicitly constructed maps are indeed fan-giving.

\begin{example} \label{example}
    Consider a pentagon whose vertices are labeled by $\{2, 3, 4, 5, 6\}$ in cyclic order.
    Let $K$ be the boundary of the two-sided cone over the pentagon, known as the \emph{pentagonal bipyramid}, with its two apexes being~$1$ and~$7$.
    Note that the link of~$1$ in~$K$ is the pentagon.
    Thus, the facet set of~$K$ is given explicitly as
    $$
        \{\{1, 2, 3\}, \{1, 2, 6\}, \{1, 3, 4\}, \{1, 4, 5\}, \{1, 5, 6\}, \{2, 3, 7\}, \{2, 6, 7\}, \{3, 4, 7\}, \{4, 5, 7\}, \{5, 6, 7\}\}.
    $$
    Set the characteristic map~$\lambda$ in the form
    $$
        \lambda = \begin{bmatrix}
                    1 & 0 & 0 & a_0 & b_0 & c_0 & d_0\\
                    0 & 1 & 0 & a_1 & b_1 & c_1 & d_1 \\
                    0 & 0 & 1 & a_2 & b_2 & c_2 & d_2
                  \end{bmatrix}.
    $$
    To ensure that $\lambda$ is fan-giving, we fix an orientation of~$K$ satisfying $o(\{1,2,3\})=1$.
    This choice completely determines the orientation map $o$.
    We then explicitly compute several determinants to match the orientation:
    \begin{align*}
        d(\{1, 3, 4\}) &= \det\left(\begin{bmatrix}
                               1 & 0 & a_0 \\
                               0 & 0 & a_1 \\
                               0 & 1 & a_2
                             \end{bmatrix}\right) = -a_1 = o(\{1, 3, 4\}) = 1, \\
        d(\{1, 2, 6\}) &= \det\left(\begin{bmatrix}
                               1 & 0 & c_0 \\
                               0 & 1 & c_1 \\
                               0 & 0 & c_2
                             \end{bmatrix}\right) = c_2 = o(\{1, 2, 6\}) = -1, \\
        d(\{2, 3, 7\}) &= \det\left(\begin{bmatrix}
                               0 & 0 & d_0 \\
                               1 & 0 & d_1 \\
                               0 & 1 & d_2
                             \end{bmatrix}\right) = d_0 = o(\{2, 3, 7\}) = -1.
    \end{align*}
    According to \cite[Lemma~6.7]{Choi-Park2016}, there are five types of fan-giving characteristic maps over a pentagon up to D-J equivalence:
    $$
        \begin{gathered}
            \begin{bmatrix}
                        1 & 0 & -1 & -1 & x \\
                        0 & 1 & 1 & 0 & -1
                      \end{bmatrix},
            \begin{bmatrix}
                        1 & 0 & -1 & -1 & 0 \\
                        0 & 1 & x & x-1 & -1
                      \end{bmatrix},
            \begin{bmatrix}
                        1 & 0 & -1 & -x & 1-x \\
                        0 & 1 & 0 & -1 & -1
                      \end{bmatrix}, \\
            \begin{bmatrix}
                        1 & 0 & -1 & 0 & 1\\
                        0 & 1 & 1-x & -1 & -1
                      \end{bmatrix}, \text{ and }
            \begin{bmatrix}
                        1 & 0 & -1 & x-1 & 1 \\
                        0 & 1 & 1 & -x & -1
                      \end{bmatrix},
        \end{gathered}
    $$
    where $x$ is an integer indeterminate.
    Thus, the projection~$\proj_{1}(\lambda)$ must match one of these five forms.
    For instance, choosing the first one, we obtain
    $$
        \lambda = \begin{bmatrix}
                    1 & 0 & 0 & a_0 & b_0 & c_0 & -1\\
                    0 & 1 & 0 & -1 & -1 & c_1 & d_1 \\
                    0 & 0 & 1 & 1 & 0 & -1 & d_2
                  \end{bmatrix}.
    $$
    Now, to ensure the compatibility of orientations, we compute the following remaining nontrivial conditions:
    \begin{align*}
        (d-o)(\{2,6,7\}) &= -c_0d_2 = 0,\\
        (d-o)(\{3,4,7\}) &= a_0d_1 = 0,\\
        (d-o)(\{4,5,7\}) &= -a_0d_2 + b_0d_1 + b_0d_2 = 0, \text{ and }\\
        (d-o)(\{5,6,7\}) &= b_0c_1d_2 + b_0d_1 + c_0d_2 =0.
    \end{align*}
    Solving this nonlinear system yields five distinct solutions for characteristic maps:
    $$
        \begin{gathered}
            \begin{bmatrix}
                        1 & 0 & 0 & 0 & 0 & 0 & -1\\
                        0 & 1 & 0 & -1 & -1 & c_1 & d_1 \\
                        0 & 0 & 1 & 1 & 0 & -1 & d_2
                      \end{bmatrix},
            \begin{bmatrix}
                        1 & 0 & 0 & a_0 & b_0 & c_0 & -1\\
                        0 & 1 & 0 & -1 & -1 & c_1 & 0 \\
                        0 & 0 & 1 & 1 & 0 & -1 & 0
                      \end{bmatrix},
            \begin{bmatrix}
                        1 & 0 & 0 & 0 & 0 & c_0 & -1\\
                        0 & 1 & 0 & -1 & -1 & c_1 & d_1 \\
                        0 & 0 & 1 & 1 & 0 & -1 & 0
                      \end{bmatrix}, \\
            \begin{bmatrix}
                        1 & 0 & 0 & b_0 & b_0 & 0 & -1\\
                        0 & 1 & 0 & -1 & -1 & 0 & 0 \\
                        0 & 0 & 1 & 1 & 0 & -1 & d_2
                      \end{bmatrix}, \text{ and }
            \begin{bmatrix}
                        1 & 0 & 0 & 0 & b_0 & 0 & -1\\
                        0 & 1 & 0 & -1 & -1 & 1 & -d_2 \\
                        0 & 0 & 1 & 1 & 0 & -1 & d_2
                      \end{bmatrix},
        \end{gathered}
    $$
    where all symbols represent integer indeterminates.
    The last two characteristic maps are special cases of the second one, considering the symmetries of~$K$.
    For detailed computations and verification, refer to the document ``fanlikes\textbackslash\textbackslash3'' in our database, introduced at the end of Section~\ref{section: fans list}.
    Using Lemma~\ref{lemma: fan-givingness}, we verify that these characteristic maps are all fan-giving.
\end{example}

It is possible to compute positivity directly without fixing any projection of $\lambda$.
However, this approach would lead to redundant computations, since many equations involving links appear repeatedly during the determination of positive characteristic maps.
For instance, the condition $d(\{1,4,5\}) = -b_2-a_2b_1 = 1$ from the previous example naturally arises when computing positivity for the pentagon link.
Thus, fixing a projection first significantly reduces such redundancies, greatly improving computational efficiency.

To compute all fan-giving characteristic maps over a given simplicial complex, we therefore utilize previously computed characteristic maps over its links.
This approach can be viewed as the inductive step of a computational induction process.
Consequently, initial data is required to start this induction.
The first non-trivial links are precisely the $1$-links, and every $1$-link is combinatorially isomorphic to the boundary of an $m$-gon, denoted by $\partial P^2(m)$, for some integer $m \geq 3$.
The fan-giving characteristic maps over $\partial P^2(m)$ have been extensively studied in the literatures, with notable results such as those in~\cite[Theorem~1.28]{Oda1988}.
To ensure this paper is self-contained, we briefly recall some of these known results.

Let $K$ be a simplicial complex isomorphic to $\partial P^2(m)$ with vertex set $\{v_1, v_2, \dots, v_m\}$ arranged in cyclic order, and let $\lambda$ be a fan-giving characteristic map over $K$.
Consider an edge $e$ of $K$ with two end points $v_i, v_{i+1}$ of $K$, and construct a new simplicial complex $K'$ by adding a new vertex $v_e$ subdividing this edge $e$.
The resulting complex $K'$ is then isomorphic to $\partial P^2(m+1)$.
Define a characteristic map $\lambda'$ over $K'$ by
$$
    \lambda_v = \begin{cases}
            \lambda_{v_i} + \lambda_{v_{i+1}}, & \mbox{if } v = v_e \\
            \lambda_{v}, & \mbox{otherwise}.
        \end{cases}
$$
The resulting pair $(K', \lambda')$ is fan-giving.
This operation is essentially a special case of the \emph{non-singular stellar subdivision} introduced by Ewald, originally defined for arbitrary fan-giving characteristic pairs; see \cite{ewald1996combinatorial}.
Indeed, the construction above coincides explicitly with the general definition when the underlying simplicial complex $K$ is restricted to the boundary of a polygon.
The following classical result provides a complete classification of toric manifolds of complex dimension~$2$.

\begin{theorem} \label{theorem: blowup}
    Every fan-giving characteristic pair $( \partial P^2(m), \lambda)$, for $m \geq 5$, is obtained by a non-singular stellar subdivision from a simpler pair.
\end{theorem}

In this paper, we focus exclusively on simplicial complexes $K$ with Picard number~$4$.
Therefore, any $1$-link in $K$ must have Picard number at most $4$.
Equivalently, each $1$-link has at most $6$ vertices.

It is well-known that a toric manifold whose underlying complex is the boundary of a quadrilateral is precisely a Hirzebruch surface.
In the language of characteristic pairs, every fan-giving characteristic map over $\partial P^2(4)$, up to fan isomorphism, is represented by
$\begin{bmatrix}
    1 & 0 & -1 & x \\
    0 & 1 & 0 & -1
\end{bmatrix},$ where $x$ is an integer indeterminate.

By applying Theorem~\ref{theorem: blowup}, we obtain the following characteristic maps over $\partial P^2(5)$:
$$
\begin{gathered}
    \begin{bmatrix}
        1 & 1 & 0  & -1 & x \\
        0 & 1 & 1 & 0 & -1
    \end{bmatrix},
    \begin{bmatrix}
    1 & 0 & -1 & -1 & x \\
    0 & 1 & 1 & 0 & -1
    \end{bmatrix},
    \begin{bmatrix}
    1 & 0 & -1 & x-1 & x \\
    0 & 1 & 0 & -1 & -1
    \end{bmatrix}, \text{ and }
    \begin{bmatrix}
    1 & 0 & -1 & x & x+1 \\
    0 & 1 & 0 & -1 & -1
    \end{bmatrix}.
\end{gathered}
$$
However, all four cases above are fan-isomorphic.
Thus, up to fan-isomorphism, there is only one type of fan over the boundary of a pentagon, represented by
$$
    \begin{bmatrix}
        1 & 0 & -1 & -1 & x \\
        0 & 1 & 1 & 0 & -1
    \end{bmatrix}.
$$

Similarly, one finds exactly three types of fans over the boundary of a hexagon:
$$
\begin{gathered}
    \begin{bmatrix}
        1 & 1 & 0 & -1 & -1 & x \\
        0 & 1 & 1 & 1 & 0 & -1
    \end{bmatrix},
    \begin{bmatrix}
        1 & 0 & -1 & -1 & -1 & x \\
        0 & 1 & 2 & 1 & 0 & -1
    \end{bmatrix}, \text{ and }
    \begin{bmatrix}
        1 & 0 & -1 & -1 & x-1 & x \\
        0 & 1 & 1 & 0 & -1 & -1
    \end{bmatrix}.
\end{gathered}
$$

\begin{remark}
As discussed previously, it is beneficial to select a link with the largest possible Picard number to minimize unknown parameters in computations.
If every $1$-link of $K$ is a triangle or a quadrilateral, then $K$ is known to be isomorphic to the boundary complex of the dual of the product of simplices (\cite{Wiemeler2015,Yu-Masuda2021}).
Such complexes are well-known to be fanlike, and their fan-giving characteristic maps are already been completely classified; see, for instance, \cite[Theorem~6.4]{Choi-Masuda-Suh2010} or \cite[Theorem~6]{Dobrinskaya2001}.
Among these, the only seed is the boundary complex of a $4$-crosspolytope (that is, the dual of a $4$-cube), which will be discussed separately in Subsection~\ref{subsection:Cube_case}.e
Therefore, in our computation, almost all remaining cases involve pentagonal or hexagonal $1$-links.
\end{remark}
%
%

\section{Complete non-singular fans over the seeds with Picard number~$4$} \label{section: fans list}
In this section, we only consider seeds that support at least one characteristic map.
All such seeds are contained in the list obtained in the authors' previous paper \cite{Choi-Jang-Vallee2024}, as summarized numerically in Theorem~\ref{theorem: classification}.
For convenience, we temporarily denote this list by $\cS$.

We say that a PL~sphere that supports a characteristic map is \emph{minimally non-fanlike} if it is not fanlike, but the link of every vertex is fanlike.
By Theorem~\ref{theorem: wedge fan-giving}, if a PL~sphere is minimally non-fanlike, then it must be a seed, and hence belongs to $\cS$.

Initially, we set $\fK$ and $\fL$ as empty sets.
We will eventually populate $\fK$ with the set of fanlike seeds, and $\fL$ with the set of minimally non-fanlike seeds, both having Picard number at most~$4$.
Based on the results of the previous section, we first place into $\fK$ all seeds in $\cS$ of dimension at most~$1$ (equivalently, $n\leq2$)  and of Picard number at most $4$.
In particular, this includes the boundary of a hexagon, denoted by $\cK^1_0 := \partial P^2(6)$, which has Picard number~$4$.

For $n\geq 3$, assume that we have already established complete information on fan-giving characteristic maps over all seeds with Picard number up to~$4$ and dimensions up to $n-2$.
Under this assumption, we examine each seed $K \in \cS$ of dimension~$n-1$ and determine whether $K$ belongs to~$\fK$, $\fL$, or neither of them, by applying the following procedure:
\begin{enumerate}
  \item If $K$ has a face whose link is in~$\fL$, then $K$ belongs to neither~$\fK$ nor~$\fL$.
  \item For each vertex $v$ of $K$ and each fan-giving characteristic map over $\lk_K(v)$, explicitly compute positive characteristic maps over $K$ as in Example~\ref{example}, and verify their fan-givingness using Lemma~\ref{lemma: fan-givingness}.
      If no fan-giving characteristic map exists, then place $K$ into~$\fL$.
  \item If $K$ successfully passed the above two steps, place it into~$\fK$.
\end{enumerate}

Step~(1) comes from Proposition~\ref{proposition: link fan-giving}, which states that every link in a fanlike PL~sphere must be fanlike.
In particular, if $K$ has a non-fanlike link, then it necessarily contains a minimally non-fanlike link.
Since the set $\fL$ already contains all lower-dimensional minimally non-fanlike seeds, this check suffices.

By applying the above procedure to every seed $K \in \cS$, we obtain a complete classification of fanlike and minimally non-fanlike seeds with Picard number at most~$4$.
As a result, the set $\fL$ consists of exactly $15$ minimally non-fanlike seeds up to simplicial isomorphism.
Their distribution is as follows:
\begin{center}
    \begin{tabular}{l*{4}{c}r}
    	\toprule
    	& \multicolumn{4}{c}{n}\tabularnewline  \cmidrule(lr){2-5}
    	$p$& $1,2,3$ & $4$ & $5$ &$>5$& total \tabularnewline \midrule
    	$2$& & &&& $0$\tabularnewline
    	$3$& & $1$& & & $1$\tabularnewline
    	$4$&  & $11$ & $3$ &  & $14$\tabularnewline
    	\bottomrule
    \end{tabular}
\end{center}
where the empty slots display zeroes.

Each minimally non-fanlike seed is denoted by $\cL^{n-1}_i$, where the superscript~$n-1$ represents the dimension and the subscript~$i$ is an index used to distinguish them.
The index starts from~$0$, following the indexing convention of our programming language.
Although a seed is often described by its set of facets, this representation is impractical in our setting due to the large number of facets involved.
To avoid listing them explicitly, we instead use the set of minimal non-faces~$\cM(K)$ for each simplicial complex~$K$, which fully determines the complex and tends to be significantly smaller.
For simplicity, each minimal non-face is represented by a string of concatenated vertex labels (e.g., the set $\{1,2,3\}$ is written as $123$) rather than using set notation.

In Table~\ref{table:minimally non-fanlike seeds}, the $15$ minimally non-fanlike seeds with Picard number at most~$4$ are represented by their sets of minimal non-faces.
In particular, $\cL^{3}_{0}$ is known as the boundary $\partial C^{4}(7)$ of the cyclic $4$-polytope with $7$ vertices, while $\cL^{3}_{7}$ and $\cL^{3}_{10}$ are known as the Br\"{u}ckner sphere~\cite{Grunbaum-Sreedharan1967} and the Barnette sphere~\cite{Barnette1973}, respectively.


\renewcommand{\arraystretch}{1.3}
\begin{table}
\centering
\begin{tabular}{cl}
\toprule
$K$ & Minimal non-faces $\cM(K)$ \\
\midrule
$\cL^3_0$ & $\{135, 136, 146, 246, 247, 257, 357\}$ \\
$\cL^3_1$ & $\{25, 57, 58, 126, 127, 168, 237, 347, 348, 468, 1346\}$ \\
$\cL^3_2$ & $\{18, 125, 126, 157, 236, 245, 246, 346, 347, 368, 378, 457, 578\}$ \\
$\cL^3_3$ & $\{37, 57, 125, 126, 127, 136, 245, 348, 368, 458, 468\}$ \\
$\cL^3_4$ & $\{26, 36, 128, 137, 138, 148, 245, 258, 347, 457, 567\}$ \\
$\cL^3_5$ & $\{15, 126, 148, 168, 237, 246, 248, 267, 347, 348, 357, 358, 567\}$ \\
$\cL^3_6$ & $\{25, 126, 127, 136, 137, 168, 268, 347, 348, 357, 457, 458, 468\}$ \\
$\cL^3_7$ & $\{126, 127, 136, 137, 138, 157, 238, 245, 256, 268, 347, 348, 457, 458, 468, 567\}$ \\
$\cL^3_8$ & $\{126, 127, 136, 137, 138, 168, 245, 256, 257, 268, 347, 348, 357, 457, 458, 468\}$ \\
$\cL^3_9$ & $\{25, 57, 58, 126, 127, 137, 268, 347, 348, 468, 1346\}$ \\
$\cL^3_{10}$ & $\{37, 125, 126, 127, 136, 168, 235, 245, 348, 457, 458, 468, 678\}$ \\
$\cL^3_{11}$ & $\{125, 126, 127, 136, 137, 157, 245, 258, 268, 347, 348, 367, 368, 457, 458, 468\}$ \\
$\cL^4_{12}$ & $\{38, 126, 246, 349, 468, 469, 479, 579, 1257, 1258, 1357\}$ \\
$\cL^4_{13}$ & $\{146, 149, 168, 249, 257, 259, 348, 349, 378, 379, 1256, 5678\}$ \\
$\cL^4_{14}$ & $\{126, 127, 179, 246, 349, 358, 468, 469, 479, 579, 1258, 1357, 2368\}$ \\
\bottomrule
\end{tabular}
\caption{Minimal non-faces of the $15$ minimally non-fanlike seeds}
\label{table:minimally non-fanlike seeds}
\end{table}
\renewcommand{\arraystretch}{1.0}

\begin{remark}
    In most cases, seeds are filtered out in step~(2) of the procedure because they do not admit any positive characteristic map compatible with those of their links.
    However, the seeds $\cL^{3}_{1}$ and $\cL^{3}_{2}$ are exceptional in that they do admit positive characteristic maps, yet none of them are fan-giving.
\end{remark}

As a consequence, we obtain the following powerful criterion for determining whether a PL~sphere with Picard number at most~$4$ is fanlike.
\begin{theorem} \label{thm:characterization_fanlike}
  Let $K$ be a PL~sphere with Picard number at most~$4$ that supports at least one characteristic map.
  Then $K$ is not fanlike if and only if it contains $\cL^{n-1}_i$ as the link of a (possibly empty) face, for some $0 \leq i \leq 14$.
\end{theorem}

We also find that the set $\fK$, obtained through the procedure, contains exactly $62$ fanlike seeds, among which $59$ have Picard number~$4$.
In particular, there exists only one fanlike seed in dimension~$7$, and no fanlike seeds exist in higher dimensions.
Each fanlike seed with Picard number~$4$ is denoted by $\cK^{n-1}_i$ for $0 \leq i \leq 58$, following the same labeling convention as for the minimally non-fanlike seeds, and is described by its set of minimal non-faces.
For clarity, when a seed has more than nine vertices, we denote the vertices~$10$, $11$, and~$12$ by the letters~$A$, $B$, and~$C$, respectively.
Their distribution is as follows:
\begin{center}
    \begin{tabular}{l*{8}{c}r}
    	\toprule
    	& \multicolumn{8}{c}{n}\tabularnewline  \cmidrule(lr){2-9}
    	$p$& $2$ & $3$ & $4$ & $5$ & $6$ &$7$& $8$& $>8$& total \tabularnewline \midrule
    	$2$& $1$ &&& &&&&& $1$\tabularnewline
    	$3$& $1$& $1$& && &&& & $2$\tabularnewline
    	$4$& $1$ & $4$& $10$& $16$ & $18$ &$9$&$1$&  & $59$\tabularnewline
    	\bottomrule
    \end{tabular}
\end{center}
where the empty slots display zeroes.

A major strength of the proposed procedure is that it not only identifies all fanlike seeds, but also provides complete information about all fan-giving characteristic maps over them.
The rest of this section is devoted to presenting all fan-giving characteristic maps over each fanlike seed with Picard number~$4$, up to fan isomorphism.
Up to D-J equivalence, we assume that the first $n$~columns of the characteristic maps form the identity matrix, and we omit these columns to save space.
Additionally, we provide the symmetric group of $\cK^{n-1}_i$, expressed as the list of vertex permutations in one-line notation that induce automorphisms of $\cK^{n-1}_{i}$.
The D-J equivalence classes over $\cK_i^{n-1}$ are then generated by the given characteristic maps and symmetries.
All symbols $x,x_0,x_1,x_2,x_3$ appearing in the characteristic maps below represent integer indeterminates.
Moreover, note that the maximum number of indeterminates is $4$ and appear in $\cM(\cK^{2}_{3})$.

\setcounter{subsection}{-1}
\subsection{$\cM(\cK^{1}_{0}) = \{13, 14, 15, 24, 25, 26, 35, 36, 46\}$}
\hfill

This is the boundary complex of a hexagon.
The toric manifolds over a sequential wedge operation starting from the boundary of a polygon (including the hexagon as a special case) are classified by Choi and Park \cite{CP_CP_variety}.
Moreover, it is shown there that all such toric manifolds are projective, generalizing the earlier results of Kleinschmidt and Sturmfels~\cite{Kleinschmidt-Sturmfels1991}.

\begin{itemize}
  \item The symmetric group consists of rotations and reflections.
  \item The fan-giving characteristic maps are
              $$\begin{gathered}
                \begin{bmatrix}
                    -1 & -1 & -1 & x \\
                    2 & 1 & 0 & -1
                  \end{bmatrix},
                \begin{bmatrix}
                    -1 & -2 & -1 & x \\
                    1 & 1 & 0 & -1
                  \end{bmatrix},
                \begin{bmatrix}
                    -1 & -1 & x-1 & x \\
                    1 & 0 & -1 & -1
                  \end{bmatrix}.
            \end{gathered}$$
\end{itemize}

\subsection{$\cM(\cK^{2}_{1}) = \{24, 36, 37, 46, 47, 57, 125, 126, 135\}$}
\begin{itemize}
  \item The symmetric group is $$ \{[1, 2, 3, 4, 5, 6, 7], [1, 5, 6, 7, 2, 3, 4]\}. $$
  \item The fan-giving characteristic maps are
            $$\begin{gathered}
                \begin{bmatrix}
                    x_0 & -1 & 0 & 1 \\
                    -1 & 0 & 1 & 2 \\
                    x_1 & -1 & -1 & -1
                \end{bmatrix},
                \begin{bmatrix}
                    x_0 & -1 & -1 & 0 \\
                    -1 & 0 & 1 & 1 \\
                    x_1 & -1 & -2 & -1
                \end{bmatrix},
                \begin{bmatrix}
                    0 & -1 & 0 & 1 \\
                    -1 & x - 2 & x - 1 & x \\
                    0 & -1 & -1 & -1
                \end{bmatrix}, \\
                \begin{bmatrix}
                    0 & -1 & -1 & 0\\
                    -1 & x - 1 & 2x - 1 & x\\
                    0 & -1 & -2 & -1
                \end{bmatrix},
                \begin{bmatrix}
                    0 & -1 & x_0 - 1 & x_0\\
                    -1 & -1 & x_1 - 1 & x_1\\
                    1 & 0 & -1 & -1
                \end{bmatrix},
                \begin{bmatrix}
                    0 & -1 & 0 & 1\\
                    -1 & x & 1 & 2\\
                    1 & -x - 1 & -1 & -1
                \end{bmatrix}, \\
                \begin{bmatrix}
                    0 & -1 & -1 & 0\\
                    -1 & -x - 2 & -x - 1 & 1\\
                    1 & x + 1 & x & -1
                \end{bmatrix},
                \begin{bmatrix}
                    0 & -1 & 0 & 1\\
                    -1 & -1 & 0 & 1\\
                    x & x - 1 & -1 & -1
                \end{bmatrix}.
            \end{gathered}$$
\end{itemize}

\subsection{$\cM(\cK^{2}_{2}) = \{14, 17, 25, 27, 36, 37, 456\}$}
\hfill

This is the boundary complex of the dual of the vertex cut of a $3$-cube.
The toric manifolds supported by this complex have been classified by Hasui et al~\cite{HKMP2020}.
This complex is particularly important because it supports the first known example of a non-projective toric manifold, originally introduced by Oda~\cite{Oda1988}.

\begin{itemize}
  \item The symmetric group is
   \begin{align*}
   \{&[1, 2, 3, 4, 5, 6, 7], [1, 3, 2, 4, 6, 5, 7], [2, 1, 3, 5, 4, 6, 7], \\ &[2, 3, 1, 5, 6, 4, 7], [3, 1, 2, 6, 4, 5, 7], [3, 2, 1, 6, 5, 4, 7]\}. \end{align*}
  \item The fan-giving characteristic maps are
            $$\begin{gathered}
              \begin{bmatrix}
                    -1 & 0 & 1 & x_0\\
                    1 & -1 & 0 & x_1\\
                    0 & 1 & -1 & -x_0 - x_1 - 1
                \end{bmatrix},
              \begin{bmatrix}
                    -1 & 0 & 0 & -1\\
                    x_0 & -1 & 0 & x_0 - 1\\
                    x_1 & x_2 & -1 & x_1 + x_2 - 1
                \end{bmatrix}, \\
              \begin{bmatrix}
                    -1 & 0 & -1 & -1\\
                    -1 & -1 & 0 & -1\\
                    0 & -1 & -1 & -1
                \end{bmatrix},
              \begin{bmatrix}
                    -1 & 0 & x & -1\\
                    1 & -1 & 0 & 0\\
                    0 & 1 & -1 & 0
                \end{bmatrix}.
            \end{gathered}$$
\end{itemize}
\subsection{$\cM(\cK^{2}_{3}) = \{17, 24, 25, 35, 36, 46\}$}
\hfill

This is the boundary complex of a pentagonal bipyramid as introduced in Example~\ref{example}.
Several specific toric manifolds supported by this complex  are discussed and classified in \cite{Choi-Park2016IJM}.

\begin{itemize}
  \item The symmetric group is the direct product of two groups: \begin{itemize}
                                                                   \item the symmetric group of the set $\{1, 7\}$,
                                                                   \item the symmetric group of the induced subcomplex on $\{2,3,4,5,6\}$ of $\cK^2_2$, which is the boundary of a pentagon.
                                                                 \end{itemize}
  \item The fan-giving characteristic maps are
  $$\begin{gathered}
    \begin{bmatrix}
                    0 & 0 & 0 & -1\\
                    -1 & -1 & x_0 & x_1\\
                    1 & 0 & -1 & x_2
                \end{bmatrix},
    \begin{bmatrix}
                    x_0 & x_1 & x_2 & -1\\
                    -1 & -1 & x_3 & 0\\
                    1 & 0 & -1 & 0
                \end{bmatrix},
    \begin{bmatrix}
                    0 & 0 & x_0 & -1\\
                    -1 & -1 & x_1 & x_2\\
                    1 & 0 & -1 & 0
                \end{bmatrix}.
  \end{gathered}$$
\end{itemize}
\subsection{$\cM(\cK^{2}_{4}) = \{25, 35, 36, 37, 47, 57, 124, 126, 146\}$}
\begin{itemize}
  \item The symmetric group is
        \begin{align*}
          \{&[1, 2, 3, 4, 5, 6, 7], [1, 2, 7, 6, 5, 4, 3], [1, 4, 3, 2, 7, 6, 5],\\
            &[1, 4, 5, 6, 7, 2, 3], [1, 6, 7, 2, 3, 4, 5], [1, 6, 5, 4, 3, 2, 7]\}.
        \end{align*}
  \item The fan-giving characteristic maps are
            $$\begin{gathered}
              \begin{bmatrix}
                    -1 & 0 & 0 & 1\\
                    -1 & 0 & 1 & 2\\
                    x & -1 & -1 & -1
                \end{bmatrix},
              \begin{bmatrix}
                    -1 & -1 & -1 & x_0\\
                    -1 & -2 & -1 & x_1\\
                    1 & 1 & 0 & -1
                \end{bmatrix},
              \begin{bmatrix}
                    -1 & 0 & 0 & 1\\
                    -1 & -1 & x - 1 & x\\
                    1 & 0 & -1 & -1
                \end{bmatrix}.
            \end{gathered}$$
\end{itemize}
\subsection{$\cM(\cK^{3}_{5}) = \{15, 26, 37, 48\}$} \label{subsection:Cube_case}
\hfill

This is the boundary complex of the $4$-crosspolytope.
A sequential wedge operation starting from $\cK^3_5$ is the boundary of the dual of the product of four simplices.
The toric manifolds over these are called generalized Bott manifolds.
Refer to \cite{Choi-Masuda-Suh2010TAMS} or \cite{Choi-Masuda-Suh2010} for more details. 

\subsection{$\cM(\cK^{3}_{6}) = \{37, 38, 57, 58, 127, 136, 245, 468, 1246\}$}
\begin{itemize}
  \item The symmetric group is \begin{align*}
                                 \{&[1, 2, 3, 4, 5, 6, 7, 8], [1, 6, 7, 4, 8, 2, 3, 5], [2, 1, 5, 6, 3, 4, 7, 8], [2, 4, 7, 6, 8, 1, 5, 3],\\
                                   &[4, 2, 8, 1, 7, 6, 5, 3], [4, 6, 5, 1, 3, 2, 8, 7], [6, 1, 8, 2, 7, 4, 3, 5], [6, 4, 3, 2, 5, 1, 8, 7] \}.
                               \end{align*}
  \item The fan-giving characteristic maps are
            $$\begin{gathered}
              \begin{bmatrix}
                    x_0 & -1 & 0 & 1\\
                    -1 & 0 & 1 & 2\\
                    x_1 & -1 & -1 & -1\\
                    -1 & 0 & 1 & 1
                \end{bmatrix},
              \begin{bmatrix}
                    0 & -1 & x_0 - 1 & x_0\\
                   -1 & -1 & x_1 - 1 & x_1\\
                    1 & 0 & -1 & -1\\
                    -1 & -1 & x_2 & x_2
                \end{bmatrix},
              \begin{bmatrix}
                    0 & -1 & 0 & 1\\
                    -1 & x & 1 & 2\\
                    1 & -x - 1 & -1 & -1\\
                    -1 & x & 1 & 1
                \end{bmatrix}.
            \end{gathered}$$
\end{itemize}
\subsection{$\cM(\cK^{3}_{7}) = \{26, 36, 37, 68, 125, 127, 348, 458, 1457\}$}
\begin{itemize}
  \item The symmetric group is trivial.
  \item The fan-giving characteristic maps are
            $$\begin{gathered}
              \begin{bmatrix}
                    -1 & -1 & -1 & x_0\\
                    -1 & -2 & -1 & x_1\\
                    1 & 1 & 0 & -1\\
                    0 & 1 & 0 & -1
                \end{bmatrix},
              \begin{bmatrix}
                    -1 & x_0 - 1 & x_0 - 1 & x_0\\
                    -1 & x_1 - 2 & x_1 - 1 & x_1\\
                    0 & -1 & -1 & -1\\
                    0 & 1 & 0 & -1
                \end{bmatrix},
              \begin{bmatrix}
                    -1 & -1 & -x - 1 & 1\\
                    -1 & -2 & -2x - 1 & 2\\
                    1 & 1 & x & -1\\
                    0 & 1 & x & -1
                \end{bmatrix}, \\
              \begin{bmatrix}
                    -1 & 0 & 0 & 1\\
                    -1 & -1 & x - 1 & x\\
                    1 & 0 & -1 & -1\\
                    0 & 0 & -1 & -1
                \end{bmatrix},
              \begin{bmatrix}
                    -1 & x_0 & 0 & 1\\
                    -1 & x_1 & 1 & 2\\
                    1 & -x_1 - 1 & -1 & -1\\
                    0 & x_2 & -1 & -1
                \end{bmatrix},
              \begin{bmatrix}
                    -1 & -x_0 - 1 & -x_0 - 1 & 1\\
                    -1 & -x_0 - 2 & -x_0 - 1 & 1\\
                    1 & x_0 + 1 & x_0 & -1\\
                    0 & x_1 + 1 & x_1 & -1
                \end{bmatrix}, \\
              \begin{bmatrix}
                    -1 & -1 & -1 & 0\\
                    -1 & -2 & -1 & 0\\
                    x_0 + 1 & 2x_0 + 1 & x_0 & -1\\
                    x_1 & 2 x_1 + 1 & x_1 & -1
                \end{bmatrix},
              \begin{bmatrix}
                    -1 & 0 & 0 & 1\\
                    -1 & -1 & 0 & 1\\
                    x & x - 1 & -1 & -1\\
                    0 & 1 & 0 & -1
                \end{bmatrix}.
            \end{gathered}$$
\end{itemize}
\subsection{$\cM(\cK^{3}_{8}) = \{26, 28, 37, 145, 148, 156, 348, 567\}$}
\begin{itemize}
  \item The symmetric group is $$ \{ [1, 2, 3, 4, 5, 6, 7, 8], [1, 2, 7, 5, 4, 8, 3, 6] \}.$$

  \item The fan-giving characteristic maps are
            $$\begin{gathered}
              \begin{bmatrix}
                    -1 & -1 & 0 & -1\\
                    0 & -1 & -1 & -1\\
                    -1 & 0 & -1 & -1\\
                    -1 & 0 & 0 & -1
                \end{bmatrix},
              \begin{bmatrix}
                    -1 & 0 & x_0 & -1\\
                    1 & -1 & 0 & 0\\
                    0 & 1 & -1 & 0\\
                    -1 & 1 & x_1 & -1
                \end{bmatrix},
              \begin{bmatrix}
                    -1 & 0 & 1 & 0\\
                    x & -1 & 0 & -1\\
                    0 & 1 & -1 & 0\\
                    -1 & 1 & 0 & 0
                \end{bmatrix}, \\
              \begin{bmatrix}
                    -1 & 0 & 0 & -1\\
                    1 & -1 & 0 & 0\\
                    x_0 & x_1 & -1 & x_0 + x_1 - 1\\
                    -1 & 1 & 0 & -1
                \end{bmatrix},
              \begin{bmatrix}
                    -1 & x_0 & 1 & x_0\\
                    0 & -1 & 0 & -1\\
                    0 & x_1 & -1 & x_1 - 1\\
                    -1 & x_2 & 0 & x_2 - 1
                \end{bmatrix},
              \begin{bmatrix}
                    -1 & 0 & 1 & 0\\
                    1 & -1 & 0 & 0\\
                    0 & x_0 & -1 & -1\\
                    -1 & x_0 & 0 & -1
                \end{bmatrix}.
            \end{gathered}$$
\end{itemize}
\subsection{$\cM(\cK^{3}_{9}) = \{25, 58, 68, 126, 127, 136, 347, 348, 457\}$}
\begin{itemize}
  \item The symmetric group is $$\{[1, 2, 3, 4, 5, 6, 7, 8], [1, 6, 7, 4, 8, 2, 3, 5]\}.$$
  \item The fan-giving characteristic maps are
            $$\begin{gathered}
              \begin{bmatrix}
                    x_0 & -1 & 0 & 1\\
                    -1 & 0 & 1 & 2\\
                    x_1 & -1 & -1 & -1\\
                    0 & 0 & -1 & -1
                \end{bmatrix},
              \begin{bmatrix}
                    1 & -1 & -1 & 0\\
                    -1 & 0 & 1 & 1\\
                    2 & -1 & -2 & -1\\
                    1 & -1 & -2 & -1
                \end{bmatrix},
              \begin{bmatrix}
                    0 & -1 & 0 & 1\\
                    -1 & x - 2 & x - 1 & x\\
                    0 & -1 & -1 & -1\\
                    0 & 0 & -1 & -1
                \end{bmatrix}, \\
              \begin{bmatrix}
                    0 & -1 & x_0 - 1 & x_0\\
                    -1 & -1 & x_1 - 1 & x_1\\
                    1 & 0 & -1 & -1\\
                    0 & 0 & -1 & -1
                \end{bmatrix},
              \begin{bmatrix}
                    0 & -1 & 0 & 1\\
                    -1 & x_0 & 1 & 2\\
                    1 & -x_0 - 1 & -1 & -1\\
                    0 & x_1 & -1 & -1
                \end{bmatrix},
              \begin{bmatrix}
                    0 & -1 & 0 & 1\\
                    -1 & -1 & 0 & 1\\
                    x_0 & x_0 - 1 & -1 & -1\\
                    x_1 & x_1 & -1 & -1
                \end{bmatrix}.
            \end{gathered}$$
\end{itemize}
\subsection{$\cM(\cK^{3}_{10}) = \{37, 68, 125, 126, 127, 136, 245, 348, 457, 458\}$}
\begin{itemize}
  \item The symmetric group is $$\{[1, 2, 3, 4, 5, 6, 7, 8], [4, 5, 3, 1, 2, 8, 7, 6]\}.$$

  \item The fan-giving characteristic map is
            $$\begin{gathered}
              \begin{bmatrix}
                    -1 & -1 & 0 & 1\\
                    -1 & 0 & 1 & 2\\
                    0 & -1 & -1 & -1\\
                    0 & 1 & 0 & -1
                \end{bmatrix}.
            \end{gathered}$$
\end{itemize}
\subsection{$\cM(\cK^{3}_{11}) = \{26, 28, 125, 145, 148, 157, 347, 348, 367, 368, 567\}$}
\begin{itemize}
  \item The symmetric group is $$\{[1, 2, 3, 4, 5, 6, 7, 8], [5, 2, 3, 7, 1, 8, 4, 6]\}.$$

  \item The fan-giving characteristic map is $$\begin{gathered} \begin{bmatrix}
                    -1 & 1 & 0 & 0\\
                    0 & -1 & 2 & 1\\
                    0 & 0 & -1 & -1\\
                    -1 & 1 & -1 & -1
                \end{bmatrix}. \end{gathered}$$
\end{itemize}
\subsection{$\cM(\cK^{3}_{12}) = \{25, 68, 126, 127, 136, 137, 347, 348, 457, 458\}$}
\begin{itemize}
  \item The symmetric group is \begin{align*} \{[1, 2, 3, 4, 5, 6, 7, 8], [1, 6, 7, 4, 8, 2, 3, 5], [4, 5, 3, 1, 2, 8, 7, 6], [4, 8, 7, 1, 6, 5, 3, 2]\}. \end{align*}

  \item The fan-giving characteristic maps are
            $$\begin{gathered}
                \begin{bmatrix}
                    x_0 & -1 & -1 & 0\\
                    -1 & 0 & 1 & 1\\
                    x_1 & -1 & -2 & -1\\
                    0 & 0 & -1 & -1
                \end{bmatrix},
                \begin{bmatrix}
                    0 & -1 & -1 & 0\\
                    -1 & x - 1 & 2x - 1 & x\\
                    0 & -1 & -2 & -1\\
                    0 & 0 & -1 & -1
                \end{bmatrix},
                \begin{bmatrix}
                    0 & -1 & x_0 - 1 & x_0\\
                    -1 & -1 & x_1 - 1 & x_1\\
                    1 & 0 & -1 & -1\\
                    0 & 0 & -1 & -1
                \end{bmatrix}.
            \end{gathered}$$
\end{itemize}
\subsection{$\cM(\cK^{3}_{13}) = \{25, 38, 57, 58, 68, 126, 347, 1247, 1346\}$}
\begin{itemize}
  \item The symmetric group is $$ \{[1, 2, 3, 4, 5, 6, 7, 8], [1, 6, 7, 4, 8, 2, 3, 5], [4, 3, 2, 1, 8, 7, 6, 5], [4, 7, 6, 1, 5, 3, 2, 8]\}.$$

  \item The fan-giving characteristic maps are
            $$\begin{gathered}
              \begin{bmatrix}
                    x_0 & -1 & 0 & 1\\
                    -1 & 0 & 1 & 2\\
                    x_1 & -1 & -1 & -1\\
                    x_2 & -1 & -1 & 0
                \end{bmatrix},
              \begin{bmatrix}
                    0 & -1 & 0 & 1\\
                    -1 & x - 2 & x - 1 & x\\
                    0 & -1 & -1 & -1\\
                    0 & -1 & -1 & 0
                \end{bmatrix},
              \begin{bmatrix}
                    0 & -1 & x_0 - 1 & x_0\\
                    -1 & -1 & x_1 - 1 & x_1\\
                    1 & 0 & -1 & -1\\
                    1 & 0 & -1 & 0
                \end{bmatrix}.
            \end{gathered}$$
\end{itemize}
\subsection{$\cM(\cK^{3}_{14}) = \{26, 58, 125, 127, 137, 157, 346, 347, 348, 468\}$}
\begin{itemize}
  \item The symmetric group is $$\{[1, 2, 3, 4, 5, 6, 7, 8], [1, 5, 3, 4, 2, 8, 7, 6], [4, 6, 7, 1, 8, 2, 3, 5], [4, 8, 7, 1, 6, 5, 3, 2]\}.$$

  \item The fan-giving characteristic map is $$\begin{gathered} \begin{bmatrix}
                    -1 & -1 & -1 & 0\\
                    -1 & -1 & 0 & 1\\
                    0 & -1 & -1 & -1\\
                    1 & 0 & 0 & -1
                \end{bmatrix}. \end{gathered}$$
\end{itemize}
\subsection{$\cM(\cK^{4}_{15}) = \{27, 146, 149, 167, 259, 358, 368, 678, 3459\}$}
\begin{itemize}
  \item The symmetric group is $$\{[1, 2, 3, 4, 5, 6, 7, 8, 9], [8, 2, 4, 3, 9, 6, 7, 1, 5]\}.$$

  \item The fan-giving characteristic maps are
            $$\begin{gathered}
              \begin{bmatrix}
                    -1 & -1 & 0 & -1\\
                    0 & -1 & -1 & -1\\
                    -1 & 0 & -1 & -1\\
                    -1 & 0 & 0 & -1\\
                    0 & 0 & -1 & -1
                \end{bmatrix},
              \begin{bmatrix}
                    -1 & 0 & x_0 & -1\\
                    1 & -1 & 0 & 0\\
                    0 & 1 & -1 & 0\\
                    -1 & 1 & x_1 & -1\\
                    1 & 0 & -1 & 0
                \end{bmatrix},
              \begin{bmatrix}
                    -1 & 0 & 1 & 0\\
                    0 & -1 & 0 & -1\\
                    0 & 1 & -1 & 0\\
                    -1 & 0 & 0 & -1\\
                    0 & 0 & -1 & -1
                \end{bmatrix}.
            \end{gathered}$$
\end{itemize}
\subsection{$\cM(\cK^{4}_{16}) = \{27, 79, 126, 349, 357, 358, 469, 1258, 1468\}$}
\begin{itemize}
  \item The symmetric group is trivial.

  \item The fan-giving characteristic maps are
            $$\begin{gathered}
            \begin{bmatrix}
                    -1 & -1 & -1 & x_0\\
                    -1 & -2 & -1 & x_1\\
                    1 & 1 & 0 & -1\\
                    0 & 1 & 0 & -1\\
                    0 & -1 & -1 & x_2
                \end{bmatrix},
            \begin{bmatrix}
                    -1 & x_0 - 1 & x_0 - 1 & x_0\\
                    -1 & x_1 - 2 & x_1 - 1 & x_1\\
                    0 & -1 & -1 & -1\\
                    0 & 1 & 0 & -1\\
                    0 & -1 & -1 & 0
                \end{bmatrix},
            \begin{bmatrix}
                    -1 & 0 & 0 & 1\\
                    -1 & 0 & 1 & 2\\
                    1 & -1 & -1 & -1\\
                    0 & 0 & -1 & -1\\
                    0 & -1 & -1 & 0
                \end{bmatrix},\\
            \begin{bmatrix}
                    -1 & -x_0 - 1 & -x_0 - 1 & 1\\
                    -1 & -x_0 - 2 & -x_0 - 1 & 1\\
                    1 & x_0 + 1 & x_0 & -1\\
                    0 & x_1 + 1 & x_1 & -1\\
                    0 & -1 & -1 & 0
                \end{bmatrix},
            \begin{bmatrix}
                    -1 & -1 & -1 & 0\\
                    -1 & -2 & -1 & 0\\
                    x_0 + 1 & 2x_0 + 1 & x_0 & -1\\
                    x_1 & 2x_1 + 1 & x_1 & -1\\
                    0 & -1 & -1 & 0
                \end{bmatrix},
            \begin{bmatrix}
                    -1 & 0 & 0 & 1\\
                    -1 & -1 & 0 & 1\\
                    x_0 & x_0 - 1 & -1 & -1\\
                    0 & 1 & 0 & -1\\
                    x_1 + 1 & x_1 & -1 & 0
                \end{bmatrix}.
            \end{gathered}$$

\end{itemize}
\subsection{$\cM(\cK^{4}_{17}) = \{27, 168, 259, 348, 349, 378, 379, 678, 1256, 1456, 1459\}$}
\begin{itemize}
  \item The symmetric group is trivial

  \item The fan-giving characteristic maps are
            $$\begin{gathered}
            \begin{bmatrix}
                    -1 & 1 & 0 & 0\\
                    0 & -1 & 2 & 1\\
                    0 & 0 & -1 & -1\\
                    -1 & 1 & -1 & -1\\
                    -1 & 0 & 1 & 0
                \end{bmatrix},
            \begin{bmatrix}
                    -1 & 0 & 0 & -1\\
                    -2 & -1 & 2 & -1\\
                    1 & 0 & -1 & 0\\
                    1 & 1 & -1 & 0\\
                    -1 & 0 & 1 & -1
                \end{bmatrix}.
            \end{gathered}$$
\end{itemize}
\subsection{$\cM(\cK^{4}_{18}) = \{27, 126, 168, 259, 348, 378, 379, 678, 1456, 1459, 3459\}$}
\begin{itemize}
  \item The symmetric group is trivial.

  \item The fan-giving characteristic maps are
            $$\begin{gathered}
            \begin{bmatrix}
                    -1 & 1 & 0 & 0\\
                    0 & -1 & 2 & 1\\
                    0 & 0 & -1 & -1\\
                    -1 & 1 & -1 & -1\\
                    -1 & 0 & 0 & -1
                \end{bmatrix}.
            \end{gathered}$$
\end{itemize}
\subsection{$\cM(\cK^{4}_{19}) = \{27, 37, 79, 126, 128, 358, 469, 3459, 14568\}$}
\begin{itemize}
  \item The symmetric group is \begin{align*} \{[1, 2, 3, 4, 5, 6, 7, 8, 9], [1, 2, 9, 5, 4, 8, 7, 6, 3]\}. \end{align*}

  \item The fan-giving characteristic maps are
            $$\begin{gathered}
            \begin{bmatrix}
                    -1 & -1 & -1 & x_0\\
                    -1 & -2 & -1 & x_1\\
                    1 & 1 & 0 & -1\\
                    0 & 1 & 0 & -1\\
                    1 & 2 & 0 & -1
                \end{bmatrix},
            \begin{bmatrix}
                    -1 & -x_0 - 1 & -x_0 - 1 & 1\\
                    -1 & -x_0 - 2 & -x_0 - 1 & 1\\
                    1 & x_0 + 1 & x_0 & -1\\
                    0 & x_1 + 1 & x_1 & -1\\
                    1 & x_0 + 2 & x_0 & -1
                \end{bmatrix},
            \begin{bmatrix}
                    -1 & -1 & -1 & 0\\
                    -1 & -2 & -1 & 0\\
                    x_0 + 1 & 2x_0 + 1 & x_0 & -1\\
                    x_1 & 2x_1 + 1 & x_1 & -1\\
                    x_2 + 1 & 2x_2 + 2 & x_2 & -1
                \end{bmatrix}.
            \end{gathered}$$
\end{itemize}
\subsection{$\cM(\cK^{4}_{20}) = \{39, 68, 256, 348, 569, 579, 1247, 1248, 1257, 1347\}$}
\begin{itemize}
  \item The symmetric group is \begin{align*} \{[1, 2, 3, 4, 5, 6, 7, 8, 9], [1, 7, 8, 4, 5, 9, 2, 3, 6]\}. \end{align*}

  \item The fan-giving characteristic maps are
            $$\begin{gathered}
            \begin{bmatrix}
                    -1 & -1 & 0 & 1\\
                    -1 & 0 & 1 & 2\\
                    0 & -1 & -1 & -1\\
                    0 & -1 & -1 & 0\\
                    -1 & 0 & 1 & 1
                \end{bmatrix},
            \begin{bmatrix}
                    0 & -1 & -1 & 0\\
                    -1 & -1 & 0 & 1\\
                    1 & 0 & -1 & -1\\
                    1 & 0 & -1 & 0\\
                    -1 & -1 & 0 & 0
                \end{bmatrix},
            \begin{bmatrix}
                    0 & -1 & 0 & 1\\
                    -1 & -1 & 0 & 1\\
                    1 & 0 & -1 & -1\\
                    0 & -1 & -1 & 0\\
                    -1 & 0 & 1 & 0
                \end{bmatrix}.
            \end{gathered}$$
\end{itemize}
\subsection{$\cM(\cK^{4}_{21}) = \{38, 79, 127, 137, 349, 469, 1256, 1258, 2456, 4568\}$}
\begin{itemize}
  \item The symmetric group is \begin{align*} \{[1, 2, 3, 4, 5, 6, 7, 8, 9], [4, 6, 3, 1, 5, 2, 9, 8, 7]\}. \end{align*}

  \item The fan-giving characteristic maps are
            $$\begin{gathered}
            \begin{bmatrix}
                    -1 & -1 & 0 & 1\\
                    -1 & 0 & 1 & 2\\
                    0 & -1 & -1 & -1\\
                    0 & 1 & 0 & -1\\
                    -1 & 0 & 0 & 1
                \end{bmatrix}.
            \end{gathered}$$
\end{itemize}
\subsection{$\cM(\cK^{4}_{22}) = \{79, 126, 127, 246, 349, 358, 468, 469, 1258, 1357\}$}
\begin{itemize}
  \item The symmetric group is \begin{align*} \{[1, 2, 3, 4, 5, 6, 7, 8, 9]\}. \end{align*}

  \item The fan-giving characteristic maps are
            $$\begin{gathered}
            \begin{bmatrix}
                    -1 & -1 & 0 & 1\\
                    -1 & 0 & 1 & 2\\
                    0 & -1 & -1 & -1\\
                    0 & 1 & 0 & -1\\
                    0 & -1 & -1 & 0
                \end{bmatrix},
            \begin{bmatrix}
                    0 & -1 & 0 & 1\\
                    -1 & -2 & -1 & 0\\
                    0 & -1 & -1 & -1\\
                    -1 & -1 & -1 & -1\\
                    0 & -1 & -1 & 0
                \end{bmatrix}.
            \end{gathered}$$
\end{itemize}
\subsection{$\cM(\cK^{4}_{23}) = \{27, 37, 38, 128, 349, 579, 1256, 1468, 4569\}$}
\begin{itemize}
  \item The symmetric group is \begin{align*} \{[1, 2, 3, 4, 5, 6, 7, 8, 9], [1, 8, 7, 5, 4, 6, 3, 2, 9]\}. \end{align*}

  \item The fan-giving characteristic maps are
            $$\begin{gathered}
            \begin{bmatrix}
                    -1 & -1 & -1 & x_0\\
                    -1 & -2 & -1 & x_1\\
                    1 & 1 & 0 & -1\\
                    0 & 1 & 0 & -1\\
                    -1 & -1 & 0 & 0
                \end{bmatrix},
            \begin{bmatrix}
                    -1 & x_0 - 1 & x_0 - 1 & x_0\\
                    -1 & x_1 - 2 & x_1 - 1 & x_1\\
                    0 & -1 & -1 & -1\\
                    0 & 1 & 0 & -1\\
                    -1 & x_2 - 1 & x_2 & x_2
                \end{bmatrix},
            \begin{bmatrix}
                    -1 & -1 & -x_0 - 1 & 1\\
                    -1 & -2 & -2x_0 - 1 & 2\\
                    1 & 1 & x_0 & -1\\
                    0 & 1 & x_0 & -1\\
                    -1 & -1 & x_1 & 0
                \end{bmatrix}, \\
            \begin{bmatrix}
                    -1 & 0 & 0 & 1\\
                    -1 & -1 & x - 1 & x\\
                    1 & 0 & -1 & -1\\
                    0 & 0 & -1 & -1\\
                    -1 & 0 & 1 & 0
                \end{bmatrix},
            \begin{bmatrix}
                    -1 & -x_0 - 1 & -x_0 - 1 & 1\\
                    -1 & -x_0 - 2 & -x_0 - 1 & 1\\
                    1 & x_0 + 1 & x_0 & -1\\
                    0 & x_1 + 1 & x_1 & -1\\
                    -1 & x_2 - 1 & x_2 & 0
                \end{bmatrix}.
            \end{gathered}$$
\end{itemize}
\subsection{$\cM(\cK^{4}_{24}) = \{127, 128, 137, 138, 256, 348, 349, 379, 468, 579, 4569\}$}
\begin{itemize}
  \item The symmetric group is \begin{align*} \{[1, 2, 3, 4, 5, 6, 7, 8, 9], [1, 2, 3, 9, 6, 5, 8, 7, 4]\}. \end{align*}

  \item The fan-giving characteristic maps are
            $$\begin{gathered}
            \begin{bmatrix}
                    0 & -1 & -1 & x_0\\
                    -1 & -2 & -1 & x_1\\
                    1 & 1 & 0 & -1\\
                    0 & 1 & 0 & -1\\
                    -1 & -1 & 0 & 0
                \end{bmatrix},
            \begin{bmatrix}
                    0 & -1 & -x_0 - 1 & 1\\
                    -1 & -2 & -2x_0 - 1 & 2\\
                    1 & 1 & x_0 & -1\\
                    0 & 1 & x_0 & -1\\
                    -1 & -1 & x_1 & 0
                \end{bmatrix},
            \begin{bmatrix}
                    0 & -1 & -1 & 0\\
                    -1 & -x_0 - 2 & -x_0 - 1 & 1\\
                    1 & x_0 + 1 & x_0 & -1\\
                    0 & x_1 + 1 & x_1 & -1\\
                    -1 & x_2 - 1 & x_2 & 0
                \end{bmatrix}.
            \end{gathered}$$
\end{itemize}
\subsection{$\cM(\cK^{4}_{25}) = \{79, 126, 257, 348, 349, 357, 358, 469, 1258, 1468\}$}
\begin{itemize}
  \item The symmetric group is \begin{align*} \{[1, 2, 3, 4, 5, 6, 7, 8, 9], [1, 6, 3, 5, 4, 2, 9, 8, 7]\}. \end{align*}

  \item The fan-giving characteristic maps are
            $$\begin{gathered}
            \begin{bmatrix}
                    -1 & -1 & -1 & x_0\\
                    -1 & -2 & -1 & x_1\\
                    1 & 1 & 0 & -1\\
                    0 & 1 & 0 & -1\\
                    0 & -1 & -1 & x_2
                \end{bmatrix},
            \begin{bmatrix}
                    -1 & -1 & -x - 1 & 1\\
                    -1 & -2 & -2x - 1 & 2\\
                    1 & 1 & x & -1\\
                    0 & 1 & x & -1\\
                    0 & -1 & -x - 1 & 1
                \end{bmatrix},
            \begin{bmatrix}
                    -1 & 0 & 0 & 1\\
                    -1 & -1 & x_0 - 1 & x_0\\
                    1 & 0 & -1 & -1\\
                    0 & 0 & -1 & -1\\
                    0 & -1 & x_1 - 1 & x_1
                \end{bmatrix}.
            \end{gathered}$$
\end{itemize}
\subsection{$\cM(\cK^{4}_{26}) = \{69, 128, 138, 168, 257, 347, 348, 349, 1256, 4579\}$}
\begin{itemize}
  \item The symmetric group is \begin{align*} \{[1, 2, 3, 4, 5, 6, 7, 8, 9], [4, 7, 8, 1, 5, 9, 2, 3, 6]\}. \end{align*}

  \item The fan-giving characteristic maps are
            $$\begin{gathered}
            \begin{bmatrix}
                    -1 & -1 & -1 & 0\\
                    -1 & -1 & 0 & 1\\
                    0 & -1 & -1 & -1\\
                    1 & 0 & 0 & -1\\
                    -1 & -1 & 0 & 0
                \end{bmatrix}.
            \end{gathered}$$
\end{itemize}
\subsection{$\cM(\cK^{4}_{27}) = \{126, 146, 149, 168, 257, 259, 348, 349, 378, 678, 3579\}$}
\begin{itemize}
  \item The symmetric group is \begin{align*} \{[1, 2, 3, 4, 5, 6, 7, 8, 9], [6, 2, 3, 8, 5, 1, 9, 4, 7]\}. \end{align*}

  \item The fan-giving characteristic maps are
            $$\begin{gathered}
            \begin{bmatrix}
                    -1 & 1 & 0 & 0\\
                    0 & -1 & 2 & 1\\
                    0 & 0 & -1 & -1\\
                    -1 & 1 & -1 & -1\\
                    0 & -1 & 1 & 0
                \end{bmatrix}.
            \end{gathered}$$
\end{itemize}
\subsection{$\cM(\cK^{4}_{28}) = \{127, 128, 137, 256, 268, 348, 349, 468, 569, 579\}$}
\begin{itemize}
  \item The symmetric group is \begin{align*} \{[1, 2, 3, 4, 5, 6, 7, 8, 9], [4, 8, 9, 5, 1, 2, 3, 6, 7], [5, 6, 7, 1, 4, 8, 9, 2, 3]\}. \end{align*}

  \item The fan-giving characteristic maps are
            $$\begin{gathered}
            \begin{bmatrix}
                    1 & -1 & -1 & 0\\
                    -1 & 0 & 1 & 1\\
                    2 & -1 & -2 & -1\\
                    1 & -1 & -2 & -1\\
                    -1 & 0 & 1 & 0
                \end{bmatrix},
            \begin{bmatrix}
                    0 & -1 & -1 & 0\\
                    -1 & -1 & 0 & 1\\
                    1 & 0 & -1 & -1\\
                    0 & 0 & -1 & -1\\
                    -1 & -1 & 0 & 0
                \end{bmatrix}.
            \end{gathered}$$
\end{itemize}
\subsection{$\cM(\cK^{4}_{29}) = \{27, 69, 126, 168, 349, 479, 1258, 1358, 3457, 3458\}$}
\begin{itemize}
  \item The symmetric group is \begin{align*} \{[1, 2, 3, 4, 5, 6, 7, 8, 9], [4, 7, 8, 1, 5, 9, 2, 3, 6]\}. \end{align*}

  \item The fan-giving characteristic maps are
            $$\begin{gathered}
            \begin{bmatrix}
                    -1 & -1 & -1 & 0\\
                    -1 & -1 & 0 & 1\\
                    0 & -1 & -1 & -1\\
                    1 & 0 & 0 & -1\\
                    0 & -1 & -1 & 0
                \end{bmatrix}.
            \end{gathered}$$
\end{itemize}
\subsection{$\cM(\cK^{4}_{30}) = \{27, 126, 128, 138, 347, 348, 479, 569, 1568, 3459\}$}
\begin{itemize}
  \item The symmetric group is \begin{align*} \{[1, 2, 3, 4, 5, 6, 7, 8, 9], [4, 7, 8, 1, 5, 9, 2, 3, 6]\}. \end{align*}

  \item The fan-giving characteristic maps are
            $$\begin{gathered}
            \begin{bmatrix}
                    -1 & -1 & -1 & 0\\
                    -1 & -1 & 0 & 1\\
                    0 & -1 & -1 & -1\\
                    1 & 0 & 0 & -1\\
                    0 & 1 & 0 & -1
                \end{bmatrix}.
            \end{gathered}$$
\end{itemize}
\subsection{$\cM(\cK^{5}_{31}) = \{79, 257, 349, 36A, 57A, 1248, 1249, 1258, 568A, 13468\}$}
\begin{itemize}
  \item The symmetric group is trivial.

  \item The fan-giving characteristic maps are
            $$\begin{gathered}
            \begin{bmatrix}
                    0 & -1 & -1 & 0\\
                    -1 & -1 & -1 & 1\\
                    2 & 1 & 0 & -1\\
                    1 & 0 & -1 & 0\\
                    -1 & -1 & 0 & 0\\
                    1 & 0 & 0 & -1
                \end{bmatrix},
            \begin{bmatrix}
                    0 & -1 & -1 & 0\\
                    -1 & -1 & 0 & 1\\
                    1 & 0 & -1 & -1\\
                    1 & 0 & -1 & 0\\
                    -1 & -1 & 0 & 0\\
                    1 & 0 & 0 & -1
                \end{bmatrix},
            \begin{bmatrix}
                    0 & -1 & 0 & 1\\
                    -1 & -1 & 0 & 1\\
                    1 & 0 & -1 & -1\\
                    0 & -1 & -1 & 0\\
                    -1 & 0 & 1 & 0\\
                    1 & 0 & 0 & -1
                \end{bmatrix}.
            \end{gathered}$$
\end{itemize}
\subsection{$\cM(\cK^{5}_{32}) = \{8A, 128, 34A, 359, 47A, 1267, 1358, 2467, 4679, 12569\}$}
\begin{itemize}
  \item The symmetric group is trivial.

  \item The fan-giving characteristic maps are
            $$\begin{gathered}
            \begin{bmatrix}
                    -1 & -1 & 0 & 1\\
                    -1 & 0 & 1 & 2\\
                    0 & -1 & -1 & -1\\
                    0 & 1 & 0 & -1\\
                    0 & -1 & -1 & 0\\
                    -1 & 0 & 0 & 1
                \end{bmatrix},
            \begin{bmatrix}
                    0 & -1 & 0 & 1\\
                    -1 & -2 & -1 & 0\\
                    0 & -1 & -1 & -1\\
                    -1 & -1 & -1 & -1\\
                    0 & -1 & -1 & 0\\
                    -1 & -1 & -1 & 0
                \end{bmatrix}.
            \end{gathered}$$
\end{itemize}
\subsection{$\cM(\cK^{5}_{33}) = \{28, 8A, 127, 368, 47A, 1269, 345A, 3569, 14579\}$}
\begin{itemize}
  \item The symmetric group is \begin{align*} \{[1, 2, 3, 4, 5, 6, 7, 8, 9, A], [4, A, 6, 1, 9, 3, 7, 8, 5, 2]\}. \end{align*}

  \item The fan-giving characteristic maps are
            $$\begin{gathered}
            \begin{bmatrix}
                    -1 & -1 & -1 & x_0\\
                    -1 & -2 & -1 & x_1\\
                    1 & 1 & 0 & -1\\
                    0 & 1 & 0 & -1\\
                    1 & 2 & 0 & -1\\
                    0 & -1 & -1 & x_2
                \end{bmatrix},
            \begin{bmatrix}
                    -1 & x_0 - 1 & x_0 - 1 & x_0\\
                    -1 & x_1 - 2 & x_1 - 1 & x_1\\
                    0 & -1 & -1 & -1\\
                    0 & 1 & 0 & -1\\
                    0 & 0 & -1 & -1\\
                    0 & -1 & -1 & 0
                \end{bmatrix},
            \begin{bmatrix}
                    -1 & 0 & 0 & 1\\
                    -1 & -1 & 0 & 1\\
                    x_0 & x_0 - 1 & -1 & -1\\
                    0 & 1 & 0 & -1\\
                    x_1 & x_1 & -1 & -1\\
                    x_2 & x_2 - 1 & -1 & 0
                \end{bmatrix}.
            \end{gathered}$$
\end{itemize}
\subsection{$\cM(\cK^{5}_{34}) = \{8A, 127, 128, 247, 47A, 1259, 1358, 346A, 3569, 4679\}$}
\begin{itemize}
  \item The symmetric group is \begin{align*} \{[1, 2, 3, 4, 5, 6, 7, 8, 9, A], [4, 7, 3, 1, 6, 5, 2, A, 9, 8]\}. \end{align*}

  \item The fan-giving characteristic maps are
            $$\begin{gathered}
            \begin{bmatrix}
                    -1 & -1 & 0 & 1\\
                    -1 & 0 & 1 & 2\\
                    0 & -1 & -1 & -1\\
                    0 & 1 & 0 & -1\\
                    0 & -1 & -1 & 0\\
                    0 & 0 & -1 & -1
                \end{bmatrix}.
            \end{gathered}$$
\end{itemize}
\subsection{$\cM(\cK^{5}_{35}) = \{258, 349, 34A, 358, 359, 38A, 68A, 1259, 1267, 1479, 467A\}$}
\begin{itemize}
  \item The symmetric group is \begin{align*} \{[1, 2, 3, 4, 5, 6, 7, 8, 9, A], [7, 6, 3, 9, A, 2, 1, 8, 4, 5]\}. \end{align*}

  \item The fan-giving characteristic maps are
            $$\begin{gathered}
            \begin{bmatrix}
                    -1 & -1 & -1 & x_0\\
                    -1 & -2 & -1 & x_1\\
                    1 & 1 & 0 & -1\\
                    0 & 1 & 0 & -1\\
                    0 & -1 & -1 & x_2\\
                    -1 & -1 & 0 & 0
                \end{bmatrix},
            \begin{bmatrix}
                    -1 & -1 & -x_0 - 1 & 1\\
                    -1 & -2 & -2x_0 - 1 & 2\\
                    1 & 1 & x_0 & -1\\
                    0 & 1 & x_0 & -1\\
                    0 & -1 & -x_0 - 1 & 1\\
                    -1 & -1 & x_1 & 0
                \end{bmatrix},
            \begin{bmatrix}
                    -1 & x_0 & 0 & 1\\
                    -1 & x_1 & 1 & 2\\
                    1 & -x_1 - 1 & -1 & -1\\
                    0 & x_2 & -1 & -1\\
                    0 & -1 & 0 & 1\\
                    -1 & x_1 + 1 & 1 & 0
                \end{bmatrix}.
            \end{gathered}$$
\end{itemize}
\subsection{$\cM(\cK^{5}_{36}) = \{179, 258, 349, 34A, 389, 789, 1267, 1467, 146A, 256A, 358A\}$}
\begin{itemize}
  \item The symmetric group is trivial.

  \item The fan-giving characteristic maps are
            $$\begin{gathered}
            \begin{bmatrix}
                    -1 & 1 & 0 & 0\\
                    0 & -1 & 2 & 1\\
                    0 & 0 & -1 & -1\\
                    -1 & 1 & -1 & -1\\
                    0 & -1 & 1 & 0\\
                    -1 & 0 & 1 & 0
                \end{bmatrix},
            \begin{bmatrix}
                    -1 & 0 & 0 & -1\\
                    -2 & -1 & 2 & -1\\
                    1 & 0 & -1 & 0\\
                    1 & 1 & -1 & 0\\
                    -1 & -1 & 1 & -1\\
                    -1 & 0 & 1 & -1
                \end{bmatrix}.
            \end{gathered}$$
\end{itemize}
\subsection{$\cM(\cK^{5}_{37}) = \{128, 129, 257, 279, 349, 479, 57A, 1368, 346A, 568A\}$}
\begin{itemize}
  \item The symmetric group is \begin{align*} \{[1, 2, 3, 4, 5, 6, 7, 8, 9, A], [4, 9, A, 5, 1, 6, 2, 3, 7, 8], [5, 7, 8, 1, 4, 6, 9, A, 2, 3]\}. \end{align*}

  \item The fan-giving characteristic maps are
            $$\begin{gathered}
            \begin{bmatrix}
                    1 & -1 & -1 & 0\\
                    -1 & 0 & 1 & 1\\
                    2 & -1 & -2 & -1\\
                    1 & -1 & -2 & -1\\
                    -1 & 0 & 1 & 0\\
                    1 & -1 & -1 & -1
                \end{bmatrix},
            \begin{bmatrix}
                    0 & -1 & -1 & 0\\
                    -1 & -1 & 0 & 1\\
                    1 & 0 & -1 & -1\\
                    0 & 0 & -1 & -1\\
                    -1 & -1 & 0 & 0\\
                    1 & 0 & 0 & -1
                \end{bmatrix}.
            \end{gathered}$$
\end{itemize}
\subsection{$\cM(\cK^{5}_{38}) = \{7A, 127, 179, 268, 34A, 48A, 1359, 3459, 12569, 34568\}$}
\begin{itemize}
  \item The symmetric group is \begin{align*} \{[1, 2, 3, 4, 5, 6, 7, 8, 9, A], [4, 8, 9, 1, 5, 6, A, 2, 3, 7]\}. \end{align*}

  \item The fan-giving characteristic maps are
            $$\begin{gathered}
            \begin{bmatrix}
                    -1 & -1 & -1 & 0\\
                    -1 & -1 & 0 & 1\\
                    0 & -1 & -1 & -1\\
                    1 & 0 & 0 & -1\\
                    0 & -1 & -1 & 0\\
                    0 & -1 & 0 & 1
                \end{bmatrix}.
            \end{gathered}$$
\end{itemize}
\subsection{$\cM(\cK^{5}_{39}) = \{25A, 268, 349, 34A, 38A, 1257, 1457, 145A, 1479, 1679, 3689, 6789\}$}
\begin{itemize}
  \item The symmetric group is trivial.

  \item The fan-giving characteristic maps are
            $$\begin{gathered}
            \begin{bmatrix}
                    -1 & 0 & 0 & -1\\
                    -2 & -1 & 2 & -1\\
                    1 & 0 & -1 & 0\\
                    1 & 1 & -1 & 0\\
                    -1 & 0 & 1 & -1\\
                    -1 & -1 & 0 & -1
                \end{bmatrix}.
            \end{gathered}$$
\end{itemize}
\subsection{$\cM(\cK^{5}_{40}) = \{7A, 179, 268, 34A, 1259, 1267, 1359, 3458, 3459, 468A\}$}
\begin{itemize}
  \item The symmetric group is \begin{align*} \{[1, 2, 3, 4, 5, 6, 7, 8, 9, A], [4, 8, 9, 1, 5, 6, A, 2, 3, 7]\}. \end{align*}

  \item The fan-giving characteristic maps are
            $$\begin{gathered}
            \begin{bmatrix}
                    -1 & -1 & -1 & 0\\
                    -1 & -1 & 0 & 1\\
                    0 & -1 & -1 & -1\\
                    1 & 0 & 0 & -1\\
                    0 & -1 & -1 & 0\\
                    -1 & -1 & 0 & 0
                \end{bmatrix}.
            \end{gathered}$$
\end{itemize}
\subsection{$\cM(\cK^{5}_{41}) = \{128, 129, 138, 257, 279, 58A, 1349, 3469, 346A, 4679, 567A\}$}
\begin{itemize}
  \item The symmetric group is trivial.

  \item The fan-giving characteristic maps are
            $$\begin{gathered}
            \begin{bmatrix}
                    1 & -1 & -1 & 0\\
                    -1 & 0 & 1 & 1\\
                    2 & -1 & -2 & -1\\
                    1 & -1 & -2 & -1\\
                    -1 & 0 & 1 & 0\\
                    0 & 0 & -1 & -1
                \end{bmatrix},
            \begin{bmatrix}
                    0 & -1 & -1 & 0\\
                    -1 & -1 & 0 & 1\\
                    1 & 0 & -1 & -1\\
                    0 & 0 & -1 & -1\\
                    -1 & -1 & 0 & 0\\
                    0 & 1 & 0 & -1
                \end{bmatrix}.
            \end{gathered}$$
\end{itemize}
\subsection{$\cM(\cK^{5}_{42}) = \{28, 127, 48A, 67A, 1259, 1359, 1679, 3458, 3459, 346A\}$}
\begin{itemize}
  \item The symmetric group is \begin{align*} \{[1, 2, 3, 4, 5, 6, 7, 8, 9, A], [4, 8, 9, 1, 5, 6, A, 2, 3, 7]\}. \end{align*}

  \item The fan-giving characteristic maps are
            $$\begin{gathered}
            \begin{bmatrix}
                    -1 & -1 & -1 & 0\\
                    -1 & -1 & 0 & 1\\
                    0 & -1 & -1 & -1\\
                    1 & 0 & 0 & -1\\
                    0 & -1 & -1 & 0\\
                    0 & 1 & 0 & -1
                \end{bmatrix}.
            \end{gathered}$$
\end{itemize}
\subsection{$\cM(\cK^{5}_{43}) = \{28, 179, 349, 389, 38A, 789, 1257, 256A, 346A, 14567, 1456A\}$}
\begin{itemize}
  \item The symmetric group is trivial.

  \item The fan-giving characteristic maps are
            $$\begin{gathered}
            \begin{bmatrix}
                    -1 & 1 & 0 & 0\\
                    0 & -1 & 2 & 1\\
                    0 & 0 & -1 & -1\\
                    -1 & 1 & -1 & -1\\
                    -1 & 0 & 1 & 0\\
                    -1 & 0 & 0 & -1
                \end{bmatrix}.
            \end{gathered}$$
\end{itemize}
\subsection{$\cM(\cK^{5}_{44}) = \{25A, 268, 349, 34A, 389, 38A, 1457, 145A, 1479, 1679, 6789, 12567\}$}
\begin{itemize}
  \item The symmetric group is \begin{align*} \{[1, 2, 3, 4, 5, 6, 7, 8, 9, A], [7, 2, 3, 9, 6, 5, 1, A, 4, 8]\}. \end{align*}

  \item The fan-giving characteristic maps are
            $$\begin{gathered}
            \begin{bmatrix}
                    -1 & 1 & 0 & 0\\
                    0 & -1 & 2 & 1\\
                    0 & 0 & -1 & -1\\
                    -1 & 1 & -1 & -1\\
                    -1 & 0 & 1 & 0\\
                    0 & -1 & 1 & 1
                \end{bmatrix}.
            \end{gathered}$$
\end{itemize}
\subsection{$\cM(\cK^{5}_{45}) = \{179, 258, 268, 34A, 67A, 1259, 1267, 1359, 3458, 3459, 468A\}$}
\begin{itemize}
  \item The symmetric group is \begin{align*} \{[1, 2, 3, 4, 5, 6, 7, 8, 9, A], [4, 8, 9, 1, 5, 6, A, 2, 3, 7]\}. \end{align*}

  \item The fan-giving characteristic maps are
            $$\begin{gathered}
            \begin{bmatrix}
                    -1 & -1 & -1 & 0\\
                    -1 & -1 & 0 & 1\\
                    0 & -1 & -1 & -1\\
                    1 & 0 & 0 & -1\\
                    0 & -1 & -1 & 0\\
                    -1 & -1 & 0 & 0
                \end{bmatrix}.
            \end{gathered}$$
\end{itemize}
\subsection{$\cM(\cK^{5}_{46}) = \{127, 147, 14A, 25A, 349, 34A, 1679, 2568, 358A, 3689, 6789\}$}
\begin{itemize}
  \item The symmetric group is trivial.

  \item The fan-giving characteristic maps are
            $$\begin{gathered}
            \begin{bmatrix}
                    -1 & 0 & 0 & -1\\
                    -2 & -1 & 2 & -1\\
                    1 & 0 & -1 & 0\\
                    1 & 1 & -1 & 0\\
                    -1 & -1 & 1 & -1\\
                    -1 & -1 & 0 & -1
                \end{bmatrix}.
            \end{gathered}$$
\end{itemize}
\subsection{$\cM(\cK^{5}_{47}) = \{257, 36A, 38A, 58A, 679, 1248, 1249, 1258, 1348, 1349, 3469, 567A\}$}
\begin{itemize}
  \item The symmetric group is \begin{align*} \{[1, 2, 3, 4, 5, 6, 7, 8, 9, A], [4, 9, 8, 1, 6, 5, 7, 3, 2, A]\}. \end{align*}

  \item The fan-giving characteristic maps are
            $$\begin{gathered}
            \begin{bmatrix}
                    0 & -1 & -1 & 0\\
                    -1 & -1 & 0 & 1\\
                    1 & 0 & -1 & -1\\
                    1 & 0 & -1 & 0\\
                    -1 & -1 & 0 & 0\\
                    0 & 1 & 0 & -1
                \end{bmatrix}.
            \end{gathered}$$
\end{itemize}
\subsection{$\cM(\cK^{5}_{48}) = \{28, 7A, 127, 48A, 1259, 1679, 3458, 346A, 13569, 34569\}$}
\begin{itemize}
  \item The symmetric group is \begin{align*} \{[1, 2, 3, 4, 5, 6, 7, 8, 9, A], [1, 7, 3, 4, 6, 5, 2, A, 9, 8], [4, 8, 9, 1, 5, 6, A, 2, 3, 7], [4, A, 9, 1, 6, 5, 8, 7, 3, 2]\}. \end{align*}

  \item The fan-giving characteristic maps are
            $$\begin{gathered}
            \begin{bmatrix}
                    -1 & -1 & -1 & 0\\
                    -1 & -1 & 0 & 1\\
                    0 & -1 & -1 & -1\\
                    1 & 0 & 0 & -1\\
                    0 & -1 & -1 & 0\\
                    0 & 0 & -1 & -1
                \end{bmatrix}.
            \end{gathered}$$
\end{itemize}
\subsection{$\cM(\cK^{6}_{49}) = \{34A, 34B, 39A, 1468, 146B, 148A, 178A, 256B, 2579, 359B, 789A, 12678\}$}
\begin{itemize}
  \item The symmetric group is \begin{align*} \{[1, 2, 3, 4, 5, 6, 7, 8, 9, A, B], [8, 2, 3, A, 5, 7, 6, 1, B, 4, 9]\}. \end{align*}

  \item The fan-giving characteristic maps are
            $$\begin{gathered}
            \begin{bmatrix}
                    -1 & 1 & 0 & 0\\
                    0 & -1 & 2 & 1\\
                    0 & 0 & -1 & -1\\
                    -1 & 1 & -1 & -1\\
                    0 & -1 & 1 & 0\\
                    -1 & 0 & 1 & 0\\
                    0 & -1 & 1 & 1
                \end{bmatrix}.
            \end{gathered}$$
\end{itemize}
\subsection{$\cM(\cK^{6}_{50}) = \{9B, 129, 48B, 1268, 1359, 2468, 347B, 357A, 1256A, 4678A\}$}
\begin{itemize}
  \item The symmetric group is \begin{align*} \{[1, 2, 3, 4, 5, 6, 7, 8, 9, A, B], [4, 8, 3, 1, 7, 6, 5, 2, B, A, 9]\}. \end{align*}

  \item The fan-giving characteristic maps are
            $$\begin{gathered}
            \begin{bmatrix}
                    -1 & -1 & 0 & 1\\
                    -1 & 0 & 1 & 2\\
                    0 & -1 & -1 & -1\\
                    0 & 1 & 0 & -1\\
                    0 & -1 & -1 & 0\\
                    -1 & 0 & 0 & 1\\
                    0 & 0 & -1 & -1
                \end{bmatrix}.
            \end{gathered}$$
\end{itemize}
\subsection{$\cM(\cK^{6}_{51}) = \{8B, 128, 269, 49B, 178A, 347B, 1256A, 1357A, 34569, 3457A\}$}
\begin{itemize}
  \item The symmetric group is \begin{align*} \{[1, 2, 3, 4, 5, 6, 7, 8, 9, A, B], [4, 9, A, 1, 5, 6, 7, B, 2, 3, 8]\}. \end{align*}

  \item The fan-giving characteristic maps are
            $$\begin{gathered}
            \begin{bmatrix}
                    -1 & -1 & -1 & 0\\
                    -1 & -1 & 0 & 1\\
                    0 & -1 & -1 & -1\\
                    1 & 0 & 0 & -1\\
                    0 & -1 & -1 & 0\\
                    0 & -1 & 0 & 1\\
                    0 & 0 & -1 & -1
                \end{bmatrix}.
            \end{gathered}$$
\end{itemize}
\subsection{$\cM(\cK^{6}_{52}) = \{12A, 28A, 1279, 1349, 134A, 1379, 2578, 346A, 346B, 468A, 568B, 579B\}$}
\begin{itemize}
  \item The symmetric group is trivial.

  \item The fan-giving characteristic maps are
            $$\begin{gathered}
            \begin{bmatrix}
                    1 & -1 & -1 & 0\\
                    -1 & 0 & 1 & 1\\
                    2 & -1 & -2 & -1\\
                    1 & -1 & -2 & -1\\
                    -1 & 0 & 1 & 0\\
                    0 & 0 & -1 & -1\\
                    -1 & 0 & 2 & 1
                \end{bmatrix}.
            \end{gathered}$$
\end{itemize}
\subsection{$\cM(\cK^{6}_{53}) = \{269, 279, 369, 39B, 79B, 126A, 1278, 345B, 356A, 478B, 1458A\}$}
\begin{itemize}
  \item The symmetric group is \begin{align*} \{&[1, 2, 3, 4, 5, 6, 7, 8, 9, A, B], [1, 2, B, 5, 4, 7, 6, A, 9, 8, 3], [4, B, 6, 1, A, 3, 7, 8, 9, 5, 2],\\ &[4, B, 2, A, 1, 7, 3, 5, 9, 8, 6], [5, 3, 7, 1, 8, B, 6, A, 9, 4, 2], [5, 3, 2, 8, 1, 6, B, 4, 9, A, 7],\\ &[8, 7, 6, 5, A, 2, B, 4, 9, 1, 3], [8, 7, 3, A, 5, B, 2, 1, 9, 4, 6], [A, 6, 7, 4, 8, 2, 3, 5, 9, 1, B],\\ &[A, 6, B, 8, 4, 3, 2, 1, 9, 5, 7]\}. \end{align*}

  \item The fan-giving characteristic maps are
            $$\begin{gathered}
            \begin{bmatrix}
                    -1 & -1 & -1 & x_0\\
                    -1 & -2 & -1 & x_1\\
                    1 & 1 & 0 & -1\\
                    0 & 1 & 0 & -1\\
                    1 & 2 & 0 & -1\\
                    0 & -1 & -1 & x_2\\
                    -1 & -1 & 0 & 0
                \end{bmatrix}.
            \end{gathered}$$
\end{itemize}
\subsection{$\cM(\cK^{6}_{54}) = \{129, 12A, 258, 28A, 134A, 1379, 346A, 468A, 568B, 579B, 3467B\}$}
\begin{itemize}
  \item The symmetric group is \begin{align*} \{[1, 2, 3, 4, 5, 6, 7, 8, 9, A, B], [8, 2, 6, 4, 9, 3, B, 1, 5, A, 7]\}. \end{align*}

  \item The fan-giving characteristic maps are
            $$\begin{gathered}
            \begin{bmatrix}
                    1 & -1 & -1 & 0\\
                    -1 & 0 & 1 & 1\\
                    2 & -1 & -2 & -1\\
                    1 & -1 & -2 & -1\\
                    -1 & 0 & 1 & 0\\
                    0 & 0 & -1 & -1\\
                    1 & -1 & -1 & -1
                \end{bmatrix}.
            \end{gathered}$$
\end{itemize}
\subsection{$\cM(\cK^{6}_{55}) = \{34B, 35A, 39B, 48B, 79B, 1279, 1359, 468A, 1256A, 12678, 24678\}$}
\begin{itemize}
  \item The symmetric group is trivial.

  \item The fan-giving characteristic maps are
            $$\begin{gathered}
            \begin{bmatrix}
                    0 & -1 & 0 & 1\\
                    -1 & -2 & -1 & 0\\
                    0 & -1 & -1 & -1\\
                    -1 & -1 & -1 & -1\\
                    0 & -1 & -1 & 0\\
                    -1 & -1 & -1 & 0\\
                    -1 & -1 & 0 & 0
                \end{bmatrix}.
            \end{gathered}$$
\end{itemize}
\subsection{$\cM(\cK^{6}_{56}) = \{25B, 34A, 34B, 39B, 1458, 145B, 148A, 2679, 369A, 12578, 1678A, 6789A\}$}
\begin{itemize}
  \item The symmetric group is trivial.

  \item The fan-giving characteristic maps are
            $$\begin{gathered}
            \begin{bmatrix}
                    -1 & 0 & 0 & -1\\
                    -2 & -1 & 2 & -1\\
                    1 & 0 & -1 & 0\\
                    1 & 1 & -1 & 0\\
                    -1 & 0 & 1 & -1\\
                    -1 & -1 & 0 & -1\\
                    -2 & -1 & 1 & -1
                \end{bmatrix}.
            \end{gathered}$$
\end{itemize}
\subsection{$\cM(\cK^{6}_{57}) = \{269, 68B, 78B, 125A, 1268, 128A, 178A, 3459, 347B, 469B, 1357A, 3457A\}$}
\begin{itemize}
  \item The symmetric group is trivial.

  \item The fan-giving characteristic maps are
            $$\begin{gathered}
            \begin{bmatrix}
                    -1 & -1 & -1 & 0\\
                    -1 & -1 & 0 & 1\\
                    0 & -1 & -1 & -1\\
                    1 & 0 & 0 & -1\\
                    0 & -1 & -1 & 0\\
                    -1 & -1 & 0 & 0\\
                    0 & 0 & -1 & -1
                \end{bmatrix}.
            \end{gathered}$$
\end{itemize}
\subsection{$\cM(\cK^{7}_{58}) = \{34B, 34C, 1459, 145C, 149B, 258C, 36AB, 38AC, 12579, 1679B, 2678A, 679AB\}$}
\begin{itemize}
  \item The symmetric group is trivial.

  \item The fan-giving characteristic maps are
            $$\begin{gathered}
            \begin{bmatrix}
                    -1 & 0 & 0 & -1\\
                    -2 & -1 & 2 & -1\\
                    1 & 0 & -1 & 0\\
                    1 & 1 & -1 & 0\\
                    -1 & 0 & 1 & -1\\
                    -1 & -1 & 0 & -1\\
                    -2 & -1 & 1 & -1\\
                    -1 & -1 & 1 & -1
                \end{bmatrix}.
            \end{gathered}$$
\end{itemize}

The implementation of the classification procedure developed in this paper is publicly available at the second author's Github repository:
\begin{center}
\url{https://github.com/Hyeontae1112/toric_manifolds_with_Picard_number_4}
\end{center}

\section{Applications}
\subsection{Upper bound for the number of minimal non-faces of fanlike spheres}
In the last section of~\cite{Batyrev1991}, after classifying toric manifolds with Picard number~$3$, Batyrev posed several interesting questions.
For a fanlike $(n-1)$-sphere $K$ with Picard number~$p$, the conjecture was that the number of minimal non-faces of $K$ is bounded in terms of $p$.
Specifically, there exists an integer sequence $\{\cN_p \}$ such that $|\cM(K)| \leq \cN_p$.
If this conjecture holds, a natural step is to determine the exact value of $\cN_p$ for each Picard number~$p$.
Batyrev explicitly raised the following:
\begin{question} \cite[Question~7.2]{Batyrev1991}
  Is there a fanlike $(n-1)$-sphere $K$ with Picard number~$p$ satisfying the following inequality?
  \begin{equation} \label{Batyrev inequality}
    |\cM(K)| > \frac{(p-1)(p+2)}{2}.
  \end{equation}
\end{question}

The boundedness has already been positively resolved.
Define $\cN(n,p)$ to be the maximal number of minimal non-faces among fanlike $(n-1)$-spheres with Picard number~$p$, and set
$$
    \cN_p = \max_{n \geq 1} \cN(n,p).
$$
Since the wedge operation neither changes the Picard number nor the number of minimal non-faces \cite{CP_wedge_2}, it should be noted that
\begin{equation} \label{eqn:N(n+1,p)}
    \cN(n+1,p) \geq \cN(n,p).
\end{equation}
Together with the above inequality and the fact that the number of fanlike seeds is finite for any fixed Picard number  by Choi and Park \cite{CP_wedge_2}, the sequence $\{\cN_p\}$ is indeed well-defined.

Since any $1$-dimensional sphere corresponds to the boundary of a polygon, it immediately follows that $\cN(2,p) = \frac{(p-1)(p+2)}{2}$.
The aim of this section is to answer Batyrev's question.
On the one hand, we prove that this equality also holds in dimension three; namely, we show  $\cN(3,p) = \frac{(p-1)(p+2)}{2}$.
On the other hand, we provide explicit constructions demonstrating the existence of fanlike spheres for each $n \geq 4$ and $p \geq 4$ satisfying the strict inequality \eqref{Batyrev inequality}, hence suggesting that a new upper bound should be conjectured.

Let $K_1$ and $K_2$ be pure simplicial complexes.
Denote their vertex sets by $\cV(K_1)$ and $\cV(K_2)$.
Consider two facets $\sigma_1 \in K_1$ and $\sigma_2 \in K_2$, and identify them as a single subset $\sigma$, that is, let $\sigma = \sigma_1 = \sigma_2$.
Their \emph{connected sum} at facet $\sigma$ is the simplicial complex:
$$K_1 \#_\sigma K_2=(K_1 \cup K_2) \setminus \{\sigma\}.$$
Note that the connected sum of two (PL)~spheres is again a (PL)~sphere.
By the definition of a connected sum, a subset $\tau$ of~$\cV(K_i)$ that is not $\sigma$ is a face of~$K_i$ if and only if it is a face of~$K_1 \#_\sigma K_2$.
For any vertices~$v_1 \in \cV(K_1)$ and~$v_2 \in \cV(K_2)$ with $v_1, \, v_2 \not \in \sigma$, $\{v_1, \, v_2\}$ is not a face of~$K_1 \#_\sigma K_2$.
Thus, \begin{equation}\label{eq: connected sum}
            \cM(K_1 \#_\sigma K_2) = \{\sigma\} \cup \cM(K_1) \cup \cM(K_2) \cup \{\{v_1, \, v_2\} \mid v_1 \in \cV(K_1), v_2 \in \cV(K_2), v_1,v_2 \not \in \sigma\}
      \end{equation}
If $K_i$ for $i=1, \, 2$ is the boundary of a simplex, then its only minimal non-face~$\cV(K_i)$ of $K_i$ is no longer minimal.
In other cases, these unions are all disjoint.

\begin{proposition}
  Every $2$-dimensional PL~sphere $K$ with $p+3$ vertices satisfies $|\cM(K)| \leq \frac{(p-1)(p+2)}{2}$.
  In other words,
  $$
    \cN(3,p) = \frac{(p-1)(p+2)}{2}.
  $$
\end{proposition}
\begin{proof}
We prove this by induction on the Picard number of $K$.
It is known that the only $2$-sphere $K$ with Picard number~$2$ is the connected sum of the boundaries of two $3$-simplices.
By \eqref{eq: connected sum}, $K$ has two minimal non-faces: one of cardinality $2$ and one of cardinality $3$.
Hence $\cM(K)=2 \leq (2-1)(2+2)/2 = 2.$

Fix $p > 2$
Assume that for each $1< q < p$, every $2$-dimensional PL~sphere $L$ with Picard number~$q$ satisfies $$|\cM(L)| \leq \frac{(q-1)(q+2)}{2}.$$
Let $K$ be a $2$-dimensional PL~sphere with Picard number~$p$.
Since $p>2$, $K$ is not the boundary of a $3$-simplex, so every minimal non-face of $K$ has cardinality less than $4$.
Then, we consider two cases: $K$ has a minimal non-face of size $3$, or every minimal non-face of $K$ has size $2$.
\medskip

If $K$ has a minimal non-face $\sigma$ of size $3$, then $K$ is the connected sum of two PL~spheres $K_1$ with Picard number~$p_1$ and $K_2$ with Picard number~$p_2$ at $\sigma$, where $p=p_1 + p_2$.

By the induction hypothesis and \eqref{eq: connected sum}, \begin{align*}
                 |\cM(K)| &\leq 1 + \frac{(p_1-1)(p_1+2)}{2} + \frac{(p_2-1)(p_2+2)}{2} + (\cV(K_1)-3)(\cV(K_2)-3) \\
                        &= \frac{(p-1)(p+2)}{2}.
               \end{align*}
\medskip

If $K$ has only minimal non-faces of size $2$, then every minimal non-face consists of the end points of a diagonal passing through the interior of the polytope associated with $K$.
Observe that a vertex $v$ is contained in at least $4$ edges, otherwise, the other three vertices that are connected to $v$ by edges form a minimal non-face since $K$ is not a simplex.
Then
$$
    |\cM(K)| \leq \frac{|\cV(K)|(|\cV(K)|-5)}{2} =\frac{(p+3)(p-2)}{2} \leq \frac{(p-1)(p+2)}{2}.
$$
\end{proof}

Let $K$ be an $(n-1)$-dimensional PL~sphere with Picard number~$p$, and $\partial \Delta^n$ denote the boundary of an $n$-simplex that shares a facet $\sigma$ with $K$.
Then the Picard number of $K \#_\sigma \partial \Delta^n$ is $p+1$ and $|\cM(K \#_\sigma \partial \Delta^n)| = 1 + |\cM(K)| + p$.
Thus, we have
\begin{equation} \label{eqn:N(n,p+1)}
  \cN(n,p+1) \geq \cN(n,p) + (p+1).
\end{equation}

\begin{proposition} \label{prop:answer_Batyrev}
  For any $p \geq 4$ and $n \geq 4$, there is an $(n-1)$-dimensional fanlike sphere with $n+p$ vertices that satisfies~\eqref{Batyrev inequality}.
  In other words,
  $$
    \cN_p \geq \cN(n,p) > \frac{(p-1)(p+2)}{2}.
  $$
\end{proposition}
\begin{proof}
In the list of fanlike seeds in the previous section, there is a $3$-dimensional fanlike seed~$K$ with $|\cM(K)| > (4-1)(4+2)/2 = 9$.
For instance, $|\cM(K^3_{11})| = 11$.
Therefore, the theorem follows immediately by~\eqref{eqn:N(n+1,p)} and~\eqref{eqn:N(n,p+1)}.
\end{proof}

In Table~\ref{table:boundBatyrev}, we display $\cN(n,p)$ and $\cN_p$ for small $n$ and $p$.
\begin{table} [ht]
	\begin{tabular}{l*{9}{c}}
		\toprule
		$p$ & $\cN(2,p)$ & $\cN(3,p)$ & $\cN(4,p)$ & $\cN(5,p)$ & $\cN(6,p)$ & $\cN(7,p)$ & $\cN(8,p)$ & $\cdots$ & $\cN_p$ \\
		\midrule
		$2$ & 2 & 2 & 2  &2  &2  &2  &2  &  & 2 \\
		$3$ & 5 & 5 & 5 & 5 &5  &5  &5  &  & 5 \\
		$4$ & 9 & 9 & 11 & 12 &12  & 12 &12  &  &12  \\
		$5$ & 14 & 14 & $\geq 16$ & $\geq 17$ & $\geq 17$ & $\geq 17$ & $\geq 17$ &  & $\geq 17$ \\
		\bottomrule
	\end{tabular}
	\caption{$\cN(n,p)$ and $\cN_p$} \label{table:boundBatyrev}
\end{table}

\subsection{Neighborly and flag fanlike spheres}
A simplicial complex on $m$ vertices is said to be \emph{$k$-neighborly} if it has the same $(k-1)$-skeleton as the $(m-1)$-simplex on the same vertex set.
In other words, every $k$-set of the vertex set belongs to the complex.
One particular class is forled by the $\lfloor \frac{n}{2} \rfloor$-neighborly $(n-1)$-spheres, which are often simply referred to as \emph{neighborly} spheres.
Such complexes appear, for instance, in the context of Stanley's upper bound theorem~\cite{Stanley1975}.

In their paper~\cite{Gretenkort-Kleinschmidt-Sturmfels1990}, Gretenkort, Kleinschmidt, and Sturmfels investigated the existence of toric manifolds over certain neighborly polytopes.
At the end of their paper, they conjectured that, for $n \geq 4$, no neighborly $n$-polytope with $n+3$ or more vertices gives rise to a toric manifold.
The motivation behind this conjecture was the strong combinatorial complexity required for a PL~sphere to be both neighborly and fanlike.
Despite its apparent simplicity, the conjecture remained open for over three decades.

It is worth noting that for an $(n-1)$-dimensional simplicial complex with $n \leq 5$, it is neighborly if and only if no pair of vertices fails to span an edge.
Therefore, the neighborliness condition is equivalent to the absence of minimal non-faces of size two.
Among the $59$ fanlike seeds with Picard number~$4$ classified in this paper, we have identified three that are neighborly: $\cK^4_{24}$, $\cK^4_{27}$, and $\cK^4_{28}$.
Indeed, the set of minimal non-faces for each of these complexes contains no element of size two, while their dimensions are all~$4$ (that is, $n=5$).
Furthermore, according to the list of polytopal $4$-spheres with $9$ vertices obtained by Fukuda et al \cite{Fukuda2013Complete}, all three are polytopal.
These provide explicit counterexamples to the conjecture of Gretenkort–Kleinschmidt–Sturmfels.
In addition, we establish even stronger results, as presented below.

\begin{proposition} \label{prop:ans_GKS}
    There exist both projective and non-projective toric manifolds supported by neighborly polytopes.
\end{proposition}

This follows from the fact that the two toric manifolds supported by $\cK^4_{27}$ are projective, whereas the four toric manifolds supported by $\cK^4_{28}$ are non-projective.
We omit the detailed proof.

On the other hand, at the opposite end of the spectrum from neighborly spheres are the flag simplicial complexes.
A simplicial complex is said to be \emph{flag} if every minimal non-face has size two; equivalently, it is the clique complex of its $1$-skeleton.
Flag complexes are of particular interest because if a toric manifold is supported by a flag complex, then its real locus, called a \emph{real toric manifold}, admits an aspherical geometric structure~\cite{Davis-Januszkiewicz-Scott1998, Uma2004, Davis2012}.
This makes flag complexes an important subclass in the study of toric topology.

While the likelihood of a simplicial complex being flag decreases as the Picard number becomes smaller, there still exist flag complexes with Picard number~$4$.
Among the fanlike seeds classified in this paper, our list reveals that exactly three of them are both flag and fanlike: $\cK^1_0$ (the boundary of a hexagon), $\cK^2_3$ (the boundary of a pentagonal bipyramid), and $\cK^3_5$ (the boundary of a $4$-crosspolytope).

\providecommand{\bysame}{\leavevmode\hbox to3em{\hrulefill}\thinspace}
\providecommand{\MR}{\relax\ifhmode\unskip\space\fi MR }
\providecommand{\MRhref}[2]{%
  \href{http://www.ams.org/mathscinet-getitem?mr=#1}{#2}
}
\providecommand{\href}[2]{#2}

\end{document}